\documentclass[12pt,a4paper,reqno]{amsart}
\usepackage{amsmath}
\usepackage{amsfonts}
\usepackage{amssymb}
\usepackage{bm} 

\newcommand\R{{\mathbf{R}}}
\newcommand\C{{\mathbf{C}}}
\newcommand\Z{{\mathbf{Z}}}

\newcommand\e{{\mathbf{e}}}
\renewcommand\H{{\mathbf{H}}}

\renewcommand\S{{\mathcal{S}}}
\newcommand\SC{{\mathcal{SC}}}

\newcommand\T{{\mathbf{T}}}
\newcommand\E{{\mathrm{E}}}
\newcommand\ESD{{\mathrm{ESD}}}

\newcommand\Dil{{\operatorname{Dil}}}
\newcommand\WM{{\operatorname{WM}}}
\newcommand\WMC{{\operatorname{WMC}}}
\newcommand\Trans{{\operatorname{Trans}}}
\newcommand\Time{{\operatorname{Time}}}
\newcommand\Rev{{\operatorname{Rev}}}
\newcommand\Rot{{\operatorname{Rot}}}
\newcommand\Frame{{\operatorname{Frame}}}
\newcommand\Energy{{\dot{\mathcal{H}^1}}}

\newcommand\loc{{\operatorname{loc}}}
\newcommand\const{{\operatorname{const}}}

\newcommand\eps{{\varepsilon}}

\newcommand\crit{{\operatorname{crit}}}

\theoremstyle{plain}
  \newtheorem{theorem}[subsection]{Theorem}
  \newtheorem{conjecture}[subsection]{Conjecture}
  
  \newtheorem{proposition}[subsection]{Proposition}
  \newtheorem{lemma}[subsection]{Lemma}
  \newtheorem{corollary}[subsection]{Corollary}

\theoremstyle{remark}

\theoremstyle{definition}
  \newtheorem{definition}[subsection]{Definition}

\include{psfig}

\begin{document}

\title[Global regularity of wave maps VI]{Global regularity of wave maps VI.  Abstract theory of minimal-energy blowup solutions}
\author{Terence Tao}
\address{Department of Mathematics, UCLA, Los Angeles CA 90095-1555}
\email{tao@math.ucla.edu}

\vspace{-0.3in}
\begin{abstract}  In \cite{tao:heatwave}, \cite{tao:heatwave2}, \cite{tao:heatwave3}, the global regularity conjecture for wave maps from two-dimensional Minkowski space $\R^{1+2}$ to hyperbolic space $\H^m$ was reduced to the problem of constructing a minimal-energy blowup solution which is almost periodic modulo symmetries in the event that the conjecture fails.  In this paper, we show that this problem can be reduced further, to that of showing that solutions at the critical energy which are either frequency-delocalised, spatially-dispersed, or spatially-delocalised have bounded ``entropy''.  These latter facts will be demonstrated in the final paper \cite{tao:heatwave5} in this series.
\end{abstract}

\maketitle

\section{Introduction}

This paper is a technical component of a larger program \cite{tao:heatwave} to establish large data global regularity for the initial value problem for two-dimensional wave maps into hyperbolic spaces.   To explain this problem, let us set out some notation.

Let $\R^{1+2}$ be Minkowski space $\{ (t,x): t \in \R, x \in \R^2 \}$
with the usual metric $g_{\alpha \beta} x^\alpha x^\beta := -dt^2 + dx^2$.  We shall work primarily in $\R^{1+2}$, parameterised by Greek indices $\alpha, \beta = 0,1,2$, raised and lowered in the usual manner.  We also parameterise the spatial coordinates by Roman indices $i,j = 1,2$.  Fix $m \geq 1$, and let $\H = (\H^m,h)$ be the $m$-dimensional hyperbolic space of constant sectional curvature $-1$.  We observe that the Lorentz group $SO(m,1)$ acts isometrically on $\H$.

A \emph{classical wave map} is a pair $(\phi,I)$, where $I$ is a time interval, and $\phi: I \times \R^2 \to \H$ is a smooth map which differs from a constant $\phi(\infty) \in \H$ by a function Schwartz in space (embedding $\H$ in $\R^{1+m}$ to define the Schwartz space), and is a (formal)  critical point of the Lagrangian
\begin{equation}\label{lagrangian}
 \int_{\R^{1+2}} \langle \partial^\alpha \phi(t,x), \partial_\alpha \phi(t,x) \rangle_{h(\phi(t,x))}\ dt dx
\end{equation}
in the sense that $\phi$ obeys the corresponding Euler-Lagrange equation
\begin{equation}\label{cov}
 (\phi^*\nabla)^\alpha \partial_\alpha \phi = 0,
\end{equation}
where $(\phi^*\nabla)^\alpha$ is covariant differentiation on the vector bundle $\phi^*(T\H)$ 
with respect to the pull-back $\phi^*\nabla$ via $\phi$ of the Levi-Civita connection $\nabla$ on the tangent bundle $T\H$ of $\H$.  If $I = \R$, we say that the wave map is \emph{global}.  We observe that wave maps have a conserved energy
\begin{equation}\label{energy-def}
 \E(\phi) = \E(\phi[t]) := \int_{\R^2} \T_{00}(t,x)\ dx = \int_{\R^2} \frac{1}{2} |\partial_t \phi|_{h(\phi)}^2
+ \frac{1}{2} |\nabla_x \phi|_{h(\phi)}^2\ dx
\end{equation}
for all $t \in I$, where $\phi[t] := (\phi(t), \partial_t \phi(t))$ is the data of $\phi$ at $t$.  We let $\WM(I,E)$ denote the collection of all classical wave maps $(\phi,I)$ with energy less than $E$.

We record five important (and well known) symmetries of wave maps:

\begin{itemize}
\item For any $t_0 \in \R$, we have the time translation symmetry 
\begin{equation}\label{time-trans}
\Time_{t_0}: \phi(t,x) \mapsto \phi(t-t_0,x).
\end{equation}
\item For any $x_0 \in \R^2$, we have the space translation symmetry
\begin{equation}\label{space-trans}
\Trans_{x_0}: \phi(t,x) \mapsto \phi(t,x-x_0).
\end{equation}
\item Time reversal symmetry
\begin{equation}\label{time-reverse}
\Rev: \phi(t,x) \mapsto \phi(-t,x).
\end{equation}
\item For every $U \in SO(m,1)$, we have the target rotation symmetry
\begin{equation}\label{rotate}
\Rot_U: \phi(t,x) \mapsto U \circ \phi(t,x).
\end{equation}
\item For every $\lambda > 0$, we have the scaling symmetry
\begin{equation}\label{cov-scaling}
\Dil_\lambda: \phi(t,x) \mapsto \phi(\frac{t}{\lambda}, \frac{x}{\lambda})
\end{equation}
\end{itemize}
In each of these symmetries, the lifespan $I$ of the wave map transforms in the obvious manner (e.g. for \eqref{time-trans}, $I$ transforms to $I+t_0$).  Also observe that the energy \eqref{energy-def} is invariant with respect to all of these symmetries.

Define \emph{classical data} to be any pair $\Phi_0 := (\phi_0,\phi_1)$, where $\phi_0: \R^2 \to \H$ is map which is in the Schwartz space modulo constants (embedding $\H$ in $\R^{1+m}$ to define the Schwartz space), and $\phi_1: \R^2 \to T\H$ is a Schwartz map with $\phi_1(x) \in T_{\phi_0(x)} \H$ for all $x \in \R^2$; let $\S$ denote the space of all classical data, equipped with the Schwartz topology; one should view this space as a non-linear analogue of the linear Schwartz data space $\S(\R^2) \times \S(\R^2)$.  

For future reference, we note that the symmetries \eqref{space-trans}, \eqref{time-reverse}, \eqref{rotate}, \eqref{cov-scaling} induce analogous symmetries on $\S$:
\begin{align}
\Trans_{x_0}: (\phi_0(x), \phi_1(x)) &\mapsto (\phi_0(x-x_0), \phi_1(x-x_0)) \label{space-trans-data}\\
\Rev: (\phi_0(x), \phi_1(x)) &\mapsto (\phi_0(x), -\phi_1(x)) \label{time-reverse-data}\\
\Rot_U: (\phi_0(x), \phi_1(x)) &\mapsto (U\phi_0(x), dU(\phi_0(x))(\phi_1(x))) \label{rotate-data}  \\
\Dil_\lambda: (\phi_0(x), \phi_1(x)) &\mapsto (\phi_0(\frac{x}{\lambda}), \frac{1}{\lambda} \phi_1(\frac{x}{\lambda})) \label{scaling-data}
\end{align}

The purpose of this paper and the other papers \cite{tao:heatwave}, \cite{tao:heatwave2}, \cite{tao:heatwave3}, \cite{tao:heatwave5} in this series is to prove the following conjecture, stated for instance in \cite{klainerman:unified}:

\begin{conjecture}[Global regularity for wave maps]\label{conj}  For every $\Phi_0 \in \S$ there exists a unique global classical wave map $(\phi,\R)$ with $\phi[0] = \Phi_0$.
\end{conjecture}

There has been extensive prior progress on this conjecture, see for instance \cite[Chapter 6]{tao:cbms} or \cite{krieger:survey} for a survey.

In fact we will prove a stronger statement than Conjecture \ref{conj}.  In \cite{tao:heatwave2} an energy space was constructed with the following properties:

\begin{theorem}[Energy space]\label{energy-claim}\cite{tao:heatwave2}  There exists a complete metric space $\Energy$ with a continuous map $\iota: \S \to \Energy$, that obeys the following properties:
\begin{itemize}
\item[(i)] $\iota(\S)$ is dense in $\Energy$.
\item[(ii)] $\iota$ is invariant under the action \eqref{rotate-data} of the rotation group $SO(m,1)$, thus $\iota(\Rot_U \Phi) = \iota(\Phi)$ for all $\Phi \in \S$. Conversely, if $\iota(\Phi) = \iota(\Psi)$, then $\Psi = \Rot_U(\Phi)$ for some $U \in SO(m,1)$.
\item[(iii)] The actions \eqref{space-trans-data}, \eqref{time-reverse-data}, \eqref{scaling-data} on $\S$ extend to a continuous isometric action on $\Energy$ (after quotienting out by rotations as in (ii)).
\item[(iv)] The energy density $\T_{00}: \S \to L^1(\R^2)$ (defined in \eqref{energy-def}) extends to a continuous map $\T: \Energy \to L^1(\R^2)$ (again after quotienting out by rotations as in (ii)).  In particular, we have a continuous energy functional $\E: \Energy \to [0,+\infty)$.
\item[(v)] If $\Phi \in \Energy$ has zero energy, thus $\E(\Phi)=0$, then $\Phi$ is constant (or more precisely, $\Phi = \iota(p,0)$ for any $p \in \H$). 
\end{itemize}
\end{theorem}

The energy space is analogous to the linear energy space $\dot H^1(\R^2) \times L^2(\R^2)$, and its construction will be reviewed in Section \ref{engen}.  In \cite{tao:heatwave3}, the following local well-posedness claim in that space was shown:

\begin{theorem}[Large data local-wellposedness in the energy space]\label{lwp-claim}  For every time $t_0 \in \R$ and every initial data $\Phi_0 \in \Energy$ there exists a \emph{maximal lifespan} $I \subset \R$, and a \emph{maximal Cauchy development} $\phi: t \mapsto \phi[t]$ from $I \to \Energy$, which obeys the following properties:
\begin{itemize}
\item[(i)] (Local existence) $I$ is an open interval containing $t_0$.
\item[(ii)] (Strong solution) $\phi: I \to \Energy$ is continuous.
\item[(iii)] (Persistence of regularity) If $\Phi_0 = \iota( \tilde \Phi_0 )$ for some classical data $\tilde \Phi_0$, then there exists a classical wave map $(\tilde \phi, I)$ with initial data $\tilde \phi[0] = \tilde \Phi_0$ such that $\phi[t] = \iota(\tilde \phi[t])$ for all $t \in I$.  (Note that $\tilde \Phi_0$, $\tilde \phi$ are ambiguous up to the action of the rotation group $SO(m,1)$, but this ambiguity disappears after applying $\iota$.)
\item[(iv)] (Continuous dependence)  If $\Phi_{0,n}$ is a sequence of data in $\Energy$ converging to a limit $\Phi_{0,\infty}$, and $\phi^{(n)}: I_n \to \Energy$ and $\phi_\infty: I_\infty \to \Energy$ are the associated maximal Cauchy developments on the associated maximal lifespans, then for every compact subinterval $K$ of $I_\infty$, we have $K \subset I_n$ for all sufficiently large $n$, and $\phi^{(n)}$ converges uniformly to $\phi$ on $K$ in the $\Energy$ topology.
\item[(v)] (Maximality)  If $t_* \in \R$ is a finite endpoint of $I$, then $\phi(t)$ has no convergent subsequence in $\Energy$ as $t \to t_*$.
\end{itemize}
\end{theorem}

A restriction of a maximal Cauchy development to a subinterval will be referred to as an \emph{energy-class solution}.

In view of Theorem \ref{lwp-claim}, Conjecture \ref{conj} follows from

\begin{conjecture}[Global well-posedness in the energy space]\label{conj2}  Every maximal Cauchy development in Theorem \ref{lwp-claim} is global, i.e. the maximal lifespan $I$ is always equal to $\R$.
\end{conjecture}

To explain how we are to prove Conjecture \ref{conj2}, we need some further notation.  

\begin{definition}[Almost periodicity]\label{ap-def} An energy class solution $\phi: I \to \Energy$ is \emph{almost periodic} if there exists functions $N: I \to (0,+\infty)$ and $x: I \to \R^2$, and a compact set $K \subset \Energy$ such that
$$ \Dil_{N(t)} \Trans_{-x(t)} \phi[t] \in K$$
for all $t \in I$.  
\end{definition}

In \cite{tao:heatwave}, \cite{tao:heatwave2}, \cite{tao:heatwave3} a ``Liouville theorem'' was established, which asserted that all almost periodic maximal Cauchy developments necessarily had zero energy.  Thus, to establish Conjecture \ref{conj2} (and thus Conjecture \ref{conj}), it suffices to show

\begin{conjecture}[Minimal energy blowup solution]\label{minimal}  Suppose that Conjecture \ref{conj2} fails.  Then there exists an almost periodic maximal Cauchy development $\phi: I \to \Energy$ with non-zero energy.
\end{conjecture}

The current paper, together with its sequel \cite{tao:heatwave5}, will be concerned with the proof of Conjecture \ref{minimal}.

The phenomenon that failure of global well-posedness in the energy class implies the existence of an almost periodic minimal-energy blowup solution is now well established in the literature, starting with the breakthrough work of Bourgain \cite{borg:scatter}, \cite{borg:book} which introduced an induction on energy argument, which was later interpreted in \cite{gopher} as constructing an approximately almost-periodic, nearly-blowing up solution with approximately the minimal energy.  In \cite{merlekenig} a simpler and cleaner method to construct genuinely almost-periodic minimal-energy blowup solutions, based on linear and nonlinear profile decompositions, was introduced.  These constructions were initially for the energy-critical non-linear Schr\"odinger equation, but have since been extended to a variety of other critical semilinear wave and dispersive equations.  Intuitively, the key point is that if a solution at the minimal energy required for blowup is not behaving in an almost periodic manner (thus breaking up into two or more non-trivial components which are widely separated in frequency or position), then one should be able to model this solution as the superposition of two widely separated (and thus almost non-interacting) components of strictly smaller energy, which one can assume to obey good scattering hypotheses.  A suitable application of perturbation theory will then show that the minimal-energy blowup solution also enjoy scattering bounds, contradicting the blowup.

There are a number of new difficulties encountered when trying to adapt these arguments to the wave map setting.  One obvious difficulty is that the function spaces used for the local well-posedness theory (essentially from \cite{tataru:wave2}, \cite{tao:wavemap2}) are much more complicated than the Strichartz-type spaces used in the semilinear theory, and in particular the linear profile decomposition is not known in these spaces.  Another problem is that the iteration scheme used to construct solutions in \cite{tao:heatwave3} is far more intricate than the contraction mapping arguments used in the semilinear theory, so the perturbation theory becomes more complicated.  Finally, the wave map equation is not a scalar equation, and it does not make much geometric sense to talk about the superposition of two or more wave maps, or to use cutoff functions to decompose wave maps into components.  This makes quite difficult to even state a nonlinear profile decomposition, let alone try to prove such a decomposition.  

These technical difficulties will be addressed in the sequel \cite{tao:heatwave5} to this paper.  In the current paper, we shall reduce Conjecture \ref{minimal} to four simpler results (Theorems \ref{symscat}, \ref{freqbound}, \ref{spacbound}, \ref{spacdeloc}) concerning the scattering theory for wave maps, and how that theory interacts with various delocalisation hypotheses on the initial data; see Theorem \ref{mainthm} for a precise statement.  These four results will then be proven in \cite{tao:heatwave5}.

In order to properly state the four results that we are reducing Conjecture \ref{minimal} to, we need two additional concepts, namely the notions of an $(A,\mu)$-wave map $(\phi,I)$, and of the energy spectral distribution $\ESD(\phi_0,\phi_1)$ of some data $(\phi_0,\phi_1) \in \S$; we now pause to discuss these concepts.

\subsection{$(A,\mu)$-wave maps}

The first ingredient required to tackle Conjecture \ref{minimal} is an improved ``scattering'' version of the local well-posedness theory from Theorem \ref{lwp-claim}, in which the solution behaves close to linearly (in particular, being stable) on rather large time intervals, with the stability and linearity improving as one decreases the interval.

To motivate this theory, let us use as an analogy the simpler (and intensively studied) \emph{energy-critical non-linear Schr\"odinger (NLS) equation}
\begin{equation}\label{nls}
 i u_t + \Delta u = |u|^4 u
\end{equation}
in three spatial dimensions, thus $u: I \times \R^3 \to \C$ for some time interval $I$.  This equation is known to be locally well-posed in the energy space $\dot H^1(\R^3)$, in the sense that a theorem analogous to Theorem \ref{lwp-claim} holds; see e.g. \cite{cwI}.  However, one can make a more quantitative version of the local well-posedness theory as follows.  Firstly, the local solutions to \eqref{nls} constructed in \cite{cwI} are not only in the energy space, but obey the additional regularity properties\footnote{Strictly speaking, the $L^2_t \dot W^{1,6}_x$ bound requires the endpoint Strichartz estimate \cite{tao:keel} which only appeared subsequently to \cite{cwI}, but let us ignore this technicality.}
$$ \| u \|_{L^2_t \dot W^{1,6}_x(I \times \R^3)} + \| u \|_{L^{10}_{t,x}(I \times \R^3)} < \infty.$$
Furthermore, if the $L^{10}_{t,x}$ norm on $I \times \R^3$ is sufficiently small compared to the energy $E = E(u) := \int_{\R^3} \frac{1}{2} |\nabla u|^2 + \frac{1}{6} |u|^6\ dx$ of $u$, thus
$$ \| u\|_{L^{10}_{t,x}(I \times \R^3)} \leq \eta$$
for some small $\eta = \eta(E) > 0$, then the NLS equation enjoys many of the same properties as the linear equation, for instance one has persistence of regularity bounds
$$ \| u(t)\|_{H^s(\R^3)} \leq 2 \| u(t_0) \|_{H^s(\R^3)}$$
for any $t,t_0 \in I$ and any fixed $s > 1$, and one also has a stability theory in which perturbations $v(t_0)$ of the initial data $u(t_0)$ which are close in the sense that
$$ \| v(t_0) - u(t_0) \|_{\dot H^1(\R^3)} \leq \eps$$
for some small $\eps = \eps(E) > 0$ depending only on $E$ will lead to another solution $v$ of the NLS \eqref{nls} on $I \times \R^3$ with
$$ \| v(t) - u(t) \|_{\dot H^1(\R^3)} \leq 2 \| v(t_0) - u(t_0) \|_{\dot H^1(\R^3)}$$
for all $t \in I$.  (See for instance \cite{gopher} for a proof of these claims.)  The key difference between these results and the more traditional local well-posedness results of the type in Theorem \ref{lwp-claim} is that the $\eps$ parameter depends only on the energy $E$, and not on the interval $I$ or solution $u$; in particular they can potentially be extended to arbitrarily long time intervals as long as the $L^{10}_{t,x}$ norm remains small.

In practice, however, we expect the $L^{10}_{t,x}$ norm to be large (but finite).  But one can iterate the small $L^{10}_{t,x}$ theory to generate a large $L^{10}_{t,x}$ theory, by using the trivial but fundamentally important observation that the $L^{10}_{t,x}$ norm is \emph{divisible}.  By this we mean that if one has a bound $\| u \|_{L^{10}_{t,x}(I \times \R^3)} \leq M$ for some compact time interval $I$ and some finite $M > 0$, and $\eta > 0$ is any small parameter, then one can subdivide $I$ into finitely many intervals $I_1,\ldots,I_k$ such that $\| u \|_{L^{10}_{t,x}(I_j \times \R^3)} \leq \eta$ for all $1 \leq j \leq k$; furthermore, the number $k$ of such intervals is controlled by a quantity depending only on $M$ and $\eta$ (specifically, one has $k \leq (M/\eta)^{10}$).  The divisibility of the $L^{10}_{t,x}$ norm then easily allows one to generate a large $L^{10}_{t,x}$ perturbation theory, which is essential for establishing the counterpart of Conjecture \ref{minimal} for that equation; see \cite{gopher} for details.  In contrast, norms which contain $L^\infty_t$ type components, such as the energy norm $L^\infty_t \dot H^1_x$, are not divisible, thus for instance there is no easy way to derive the large energy perturbation theory directly from the small energy theory.

The key building block in the above theory was the concept of an ``almost linear'' solution: a pair $(u,I)$ which obeyed various spacetime bounds on the slab $I \times \R^3$, but for which a certain crucial and divisible norm (the $L^{10}_{t,x}$ norm was small).  For wave maps, we will abstract this concept by introducing the notion of an \emph{$(A,\mu)$-wave map} for various choices of parameters $A, \mu > 0$, where one should think of $A$ as being bounded and $\mu$ as being small. A $(A,\mu)$-wave map will be a pair $(\phi,I)$, where $I$ is a compact non-empty time interval and $\phi: I \to \S$ is a classical wave map on $I$ of energy at most $A$ obeying certain estimates relating to the parameters $A, \mu$.  The precise estimates required are complicated to state and will be deferred to the sequel \cite{tao:heatwave5} of this paper.  However, as a rough approximation, these estimates are asserting that a certain\footnote{This is a very crude caricature of the situation.  In particular, the maps $\phi: I \times \R^3 \to \H$ do not form a vector space, and so the use of the term ``norm'' is, strictly speaking, inaccurate.} ``norm'' $\| \phi \|_{S^1(I \times \R^2)}$ of the wave map is bounded by $A$, and for which a certain key ``controlling norm'' $\| \phi \|_{S^1_*(I \times \R^2)}$, which is divisible, is bounded by the smaller quantity $\mu$; the analogue for NLS would be a solution $u: I \times \R^3 \to \C$ to \eqref{nls} whose $L^\infty_t \dot H^1_x(I \times \R^3)$ and $L^2_t \dot W^{1,6}_x(I \times \R^3)$ norms were bounded by $A$, and whose $L^{10}_{t,x}(I \times \R^3)$ norm was bounded by $\mu$.  Very informally, one should think of an $(A,\mu)$-wave map as a wave map which is ``within $\mu$'' of behaving like a solution to the linear wave equation at energy $A$.

For any $A > 0$, $0 < \mu \leq 1$ and any classical wave map $(\phi,I)$ with $I$ compact, define the \emph{$(A,\mu)$-entropy} of $(\phi,I)$ to be the least number $m$ of compact intervals $I_1,\ldots,I_m \subset I$ needed to cover $I$, such that the restrictions $(\phi\downharpoonright_{I_j},I_j)$ are $(A,\mu)$-wave maps for $1 \leq j \leq m$.  (The finiteness of this entropy will be established in Corollary \ref{finsiz} below.)

In \cite{tao:heatwave5} we will show that the concept of an $(A,\mu)$-wave map defined in that paper obeys the following key properties (see Section \ref{notation-sec} below for the asymptotic notation conventions used in this paper).

\begin{theorem}[Properties of $(A,\mu)$-wave maps]\label{symscat}  For every $A, \mu > 0$ there exists a class of classical wave maps $(\phi,I)$ on compact time intervals with the following properties:
\begin{itemize}
\item[(i)] (Monotonicity) If $(\phi,I)$ is an $(A,\mu)$-wave map for some $A > 0$ and $0 < \mu \leq 1$, then it is also a $(A',\mu')$-wave map for every $A' \geq A$ and $\mu' \geq \mu$.  Also, $(\phi\downharpoonright_J,J)$ is an $(A,\mu)$-wave map for every $J \subset I$.
\item[(ii)] (Symmetries) If $(\phi,I)$ is an $(A,\mu)$-wave map for some $A > 0$ and $0 < \mu \leq 1$, then any application of the symmetries \eqref{time-trans}-\eqref{cov-scaling} to $\phi$ (and $I$) is also a $(A,\mu)$-wave map.
\item[(iii)] (Divisibility)  If $(\phi,I)$ is an $(A,\mu)$-wave map, with $\mu$ sufficiently small depending on $A$, and $0 < \mu' \leq 1$, then $(\phi,I)$ has an $(O(A),\mu')$-entropy of $O_{A,\mu,\mu'}(1)$.
\item[(iv)] (Continuity)  Let $A > 0$, and let $0 < \mu < 1$ be sufficiently small depending on $A$.  If $I$ is a compact interval, $I_n$ is an increasing sequence of compact sub-intervals which exhaust the interior of $I$, and $(\phi^{(n)},I_n)$ is a sequence of $(A,\mu)$-wave maps which converge uniformly in $\Energy$ on every compact interval of $I$, then the limit extends continuously in $\Energy$ to all of $I$.
\item[(v)] (Small data scattering) For every $0 < \mu \leq 1$ there exists an $E > 0$ such that whenever $I$ is a interval and $(\phi,I)$ is a classical wave map with energy at most $E$, then $\phi$ is a $(\mu,\mu)$-wave map.
\item[(vi)] (Local well-posedness in the scattering size) For every $E > 0$, $0 < \mu \leq 1$, $\Phi \in \Energy$, $t_0\in \R$ with $\E(\Phi) \leq E$, there exists a compact interval $I$ with $t_0$ in the interior, such that for any $(\phi_0,\phi_1) \in \S$ with $\iota(\phi_0,\phi_1)$ sufficiently close to $\Phi$ in $\Energy$, there is a $(O_E(1),\mu)$-wave map $(\phi,I)$ with $\phi[t_0] = (\phi_0,\phi_1)$.
\end{itemize}
\end{theorem}

The local well-posedness machinery in \cite{tao:heatwave3} is already sufficient to build a concept of an $(A,\mu)$-wave map that obeys all the above properties except for the crucial divisibility property (iii).  However, this property is of critical importance to the rest of the argument, and so we will not be able to directly use the theory in \cite{tao:heatwave3}, instead developing a more sophisticated version of that theory in \cite{tao:heatwave5}.

From Theorem \ref{symscat}(vi) and compactness we have

\begin{corollary}[Finiteness of the scattering size]\label{finsiz} Let $\phi: I \to \Energy$ be a maximal Cauchy development with $\E(\phi) \leq E$, and let $\phi^{(n)}: I_n \to \S$ are smooth maximal Cauchy developments that converge to $\phi$ in the sense of Theorem \ref{lwp-claim}(iv).  Let $K$ be a compact subinterval of $I$, and let $0 < \mu \leq 1$.  Then $(\phi^{(n)},K)$ has a $(O_E(1),\mu)$-entropy of $O_{\phi,K,\mu}(1)$ for all sufficiently large $n$.
\end{corollary}

Call an energy $E > 0$ \emph{good} if there exists $A,M > 0$, $0 < \mu \leq 1$ such that every classical wave map $(\phi,I)$ with $I$ compact and $\E(\phi) \leq E$ has an $(A,\mu)$-entropy of at most $M$.  Call $E$ \emph{bad} if it is not good. As we shall see in Lemma \ref{quant} below,
Conjecture \ref{minimal} will follow if we can show that every energy is good.

\subsection{Energy spectral distribution}

The second key concept we will need is that of the \emph{energy spectral distribution} $\ESD(\phi_0,\phi_1)$ of an initial data $(\phi_0,\phi_1) \in \S$.  To motivate this concept, let us again return to solutions $u$ to the NLS \eqref{nls}, and specifically to the energy $E(u)$ of such solutions.  If we ignore for sake of discussion the potential energy term $\int_{\R^3} \frac{1}{6} |u|^6\ dx$ of the energy $E(u)$ and focus instead on the kinetic energy $\int_{\R^3} \frac{1}{2} |\nabla u|^2\ dx$ (note that Sobolev embedding morally allows one to dominate the former by the latter), then we can decompose the energy into Littlewood-Paley components, obtaining the heuristic
\begin{equation}\label{eid}
 E(u) \sim \sum_k \| P_k u(t) \|_{\dot H^1(\R^3)}^2
\end{equation}
for any fixed time $t$, where $P_k u$ is a suitable Littlewood-Paley projection to frequencies $\sim 2^k$ (the precise definition of $P_k$ is not important for this discussion).  One can thus view the sequence $\|P_k u(t)\|_{\dot H^1(\R^3)}^2$ for $k \in \Z$ as describing the spectral distribution of energy of $u$ at time $t$ into low, medium, and high frequencies.

A key observation, originating in \cite{borg:scatter}, is that solutions (to \eqref{nls}) whose spectral distribution is \emph{delocalised} in the sense that there is significant energy both at very low and very high frequencies behave as if they were the non-interacting superposition of two solutions of strictly smaller energy (informally, these are the ``low-frequency component'' and ``high-frequency component'' of the original solution).  This observation (which we will mimic for wave maps via Theorem \ref{freqbound} below) is crucial in establishing the analogue of Conjecture \ref{minimal} for NLS; see \cite{gopher} for details.  It is thus desirable to define an analogue of this distribution for wave maps.

Unfortunately, the (linear, discrete) Littlewood-Paley projection operators do not easily apply directly to maps $\phi: \R^2 \to \H$ taking values in hyperbolic space for a number of reasons (such as the exponential volume growth\footnote{In a very recent preprint\cite{sterb} involving wave maps into compact targets, this problem is avoided by embedding the target into Euclidean space in which the standard Littlewood-Paley operators are available.  Hyperbolic space does not embed into Euclidean space directly; however, one can descend to a compact quotient in order to obtain the embedding, which suffices for the purposes of establishing global regularity; we thank Jacob Sterbenz for informing us of this observation, which originates from Joachim Krieger.} of $\H$).  However, it turns out that by using the harmonic map heat flow as a substitute for the linear Littlewood-Paley theory, one can obtain a (nonlinear, continuous) family of ``Littlewood-Paley projection operators'' which serve as an adequate replacement for the purposes of defining an energy spectral density.  More precisely, in Section \ref{esd-sec} below we will define the \emph{energy spectral density} $\ESD(\phi_0,\phi_1): \R^+ \to \R^+$ of a given classical pair of data $(\phi_0,\phi_1) \in \S$, and set out its basic properties.  For now, we list just a single property, namely the energy identity
\begin{equation}\label{energy-ident}
\E(\phi_0,\phi_1) = \int_0^\infty \ESD(\phi_0,\phi_1)(s)\ ds
\end{equation}
which is the counterpart to \eqref{eid}.  Indeed, $\ESD(\phi_0,\phi_1)(s)$ will measure the rate of energy dissipation in the harmonic map heat flow after evolving that flow for time $s$.  The rough analogue of $\ESD(\phi_0,\phi_1)(s)$ for NLS would be the quantity $2^{2k} \| P_k u(t) \|_{\dot H^1(\R^3)}^2$, where $k$ is an integer such that $2^{-2k} \sim s$.   

\subsection{Frequency localisation}

We can now state the three remaining results needed for our conditional proof of Conjecture \ref{minimal}.  The first, which has already been alluded to earlier, is the claim that frequency delocalised data can be controlled by solutions of strictly smaller energy.

\begin{theorem}[Frequency delocalisation implies spacetime bound]\label{freqbound}  Let $0 < E_0 < \infty$ be such that every energy strictly less than $E_0$ is good. Let $\phi^{(n)}[0] \in \S$ be a sequence of initial data with the energy bound
\begin{equation}\label{enbound}
 \E( \phi^{(n)}[0] ) \leq E_0 + o_{n \to \infty}(1)
\end{equation}
and suppose that there exists $\eps > 0$ (independent of $n$) and $K_n \to \infty$ such that we have the frequency delocalisation property
$$
 \int_{1/K_n < s < K_n} \ESD(\phi^{(n)}[0])(s)\ ds = o_{n \to \infty}(1)
$$
and
$$  \int_{s \geq K_n} \ESD(\phi^{(n)}[0])(s)\ ds, \int_{s \leq 1/K_n} \ESD(\phi^{(n)}[0])(s)\ ds \geq \eps$$
for all $n$.  Then there exists $A, M > 0$ independent of $n$ such that for each sufficiently large $n$ and every compact time interval $I \ni 0$, one can extend $\phi^{(n)}[0]$ to a classical wave map $(\phi^{(n)}, I)$ with $(A,1)$-entropy at most $M$.
\end{theorem}

This result is analogous to \cite[Proposition 4.3]{gopher} for NLS.  The proof of that proposition proceeded by decomposing (the analogue of) $\phi^{(n)}[0]$ into high and low frequency components of strictly smaller energy (plus a negligible medium-frequency error).  By the hypotheses on $E_0$, one can evolve the high and low frequency components on $I$ separately, with good spacetime bounds on both (in particular, $L^{10}_{t,x}$ bounds).  Furthermore, it turns out that high-low frequency interactions for the energy-critical NLS are negligible, so the superposition of the two solutions will approximately solve the original equation \eqref{nls}, with approximately the right initial data.  A suitable application of long-time perturbation theory then concludes the argument; see \cite{gopher} for full details.

In \cite{tao:heatwave5} we will establish this theorem by a similar argument, in particular relying heavily on the perturbation theory of $(A,\mu)$-wave maps.  However, there are several new technical complications when trying to adapt the NLS argument to wave maps.  Firstly, the non-scalar and non-linear nature of wave maps $\phi: I \times \R^2 \to \H$ means that one cannot ``decompose'' a map into components, or ``superimpose'' those components to reconstitute the original map, in as easy a fashion as can be done in the NLS setting.  This problem can be resolved (though at the cost of significantly lengthening the argument) by once again turning to the harmonic map heat flow to ``resolve'' the non-scalar wave map $\phi$ into a ``scalar'' field $\psi_s: \R^+ \times I \times \R^2 \to \R^m$, to which linear decomposition and superposition tools can be applied; see \cite{tao:heatwave5} for details.  (The strategy of reducing to the scalar field $\psi_s$ is analogous to the method of ``dynamic separation'' used in \cite{krieger:2d}.)

A potentially more serious problem is that the high-low frequency interactions are not fully negligible for wave maps.  More precisely, while the high-frequency component still has a negligible impact on the low frequency component, the converse is not true; the low frequency component sets up a non-trivial connection field (with non-trivial curvature) which distorts the evolution of the high-frequency component (roughly speaking, it causes the latter to evolve by an approximate covariant wave equation rather than by an approximate free wave equation).  However, this interaction does not actually cause any significant energy transfer between high and low frequencies (because the total solution has conserved energy, and the low frequency component, being largely unaffected by the high frequencies, also approximately obeys energy conservation).  As a consequence, it is still possible to execute the strategy in \cite{gopher} by first evolving the low frequency component, then using divisibility to decompose the time interval $I$ into smaller intervals on which the low frequency component is almost linear.  On each such interval, the high frequency component has energy strictly less than $E_0$ (by the energy conservation property) and can be evolved by appealing to the hypotheses of $E_0$ and then reconstituted with the low frequency component by perturbation theory (as well as the resolution $\phi \mapsto \psi_s$ alluded to earlier).  One then iterates this procedure along these intervals to reconstitute the whole solution.  (This strategy is somewhat analogous to the ``Fourier truncation method'' of Bourgain \cite{borg:book} to construct global solutions to sub-critical equations below the energy norm.)

\subsection{Spatial concentration}

The next result we need asserts that if an initial data is dispersed in space, then it can be evolved as if it had strictly lower energy either forward in time or backward in time.

\begin{theorem}[Spatial dispersion implies spacetime bound]\label{spacbound} Let $0 < E_0 < \infty$ be such that every energy strictly less than $E_0$ is good. Let $\phi^{(n)}[0] \in \S$ be a sequence of initial data with the energy bound \eqref{enbound}, which is uniformly localised in frequency in the sense that for every $\eps > 0$ there exists a $C > 0$ (independent of $n$) such that
\begin{equation}\label{spacloc}
  \int_{s \geq C} \ESD(\phi^{(n)}[0])(s)\ ds, \int_{s \leq 1/C} \ESD(\phi^{(n)}[0])(s)\ ds \leq \eps
\end{equation}
for all $n$.  Suppose also that the data is asymptotically spatially dispersed in the sense that
\begin{equation}\label{disp}
 \sup_{x_0 \in \R^2} \int_{|x-x_0| \leq 1} \T_{00}(\phi^{(n)})(0,x)\ dx = o_{n \to \infty}(1),
\end{equation}
where the energy density $\T_{00}$ was defined in \eqref{energy-def}.
Then there exists $A, M > 0$ independent of $n$ such that for each sufficiently large $n$ and every compact time interval $I \ni 0$, one can extend $\phi^{(n)}[0]$ to a classical wave map $(\phi^{(n)}, I_\pm)$ to either the forward time interval $I_+ := I \cap [0,+\infty)$ or the backward time interval $I_- := I \cap (-\infty,0]$ with $(A,1)$-entropy at most $M$.
\end{theorem}

This result is roughly analogous to \cite[Proposition 4.5]{gopher} for NLS.  Informally, the argument in \cite{gopher} proceeds as follows.  One first considers the linear evolution from the initial data (the analogue of $\phi^{(n)}[0]$), which will be frequency-localised by hypothesis.  By the analogue of \eqref{disp} for NLS, one can show that this linear evolution does not concentrate spatially at any time $t$ close to $0$.

Suppose that the linear solution in fact did not concentrate spatially at any time $t$.  In this case, the evolution will be ``small'' in a certain sense (indeed, the $L^{10}_{t,x}$ norm will be small) and the solution can be easily constructed (on both $I_+$ and $I_-$) by perturbation theory.  

Now suppose that the linear solution concentrated spatially at some large negative time $t < 0$.  In that case, we turn to the future portion $I_+$ of the time interval $I$.  The concentration of the linear solution at the distant past implies that one can decompose the linear solution in $I_+$ into the superposition of a dispersed component (relating to the evolution of the concentrated portion of the solution in the past), plus a component of strictly smaller energy than $E_0$.  By choice of $E_0$, the latter can be evolved by the nonlinear equation \eqref{nls} to the whole of $I_+$, and then perturbation theory can be used to paste in the dispersed component and thus recover the claim (on just $I_+$) by perturbation theory.

Finally, if the linear solution concentrated spatially at some large positive time $t > 0$, then we perform the time-reversal of the preceding argument and extend the solution to $I_-$ instead.  (See \cite{gopher} for full details.)

We will prove Theorem \ref{spacbound} in \cite{tao:heatwave5}.  To execute the above strategy we will need the nonlinear resolution $\phi \mapsto \psi_s$ alluded to earlier, as well as the perturbation theory of $(A,\mu)$-wave maps. A minor technical difficulty is that due to our definition of $(A,\mu)$-wave maps, we will not be able to ensure that the linear solution concentrates at a point in spacetime, but rather along a light ray.  This difficulty could be avoided by making the definition of such maps even more complicated, but it turns out that concentration along a light ray suffices for our application, the point being that solutions that are localised to a light ray in the distant past will still be dispersed in the future, and similarly with the roles of past and future reversed.

\subsection{Spatial localisation}

The final ingredient we need is a spatial analogue of Theorem \ref{freqbound}:

\begin{theorem}[Spatial delocalisation implies spacetime bound]\label{spacdeloc} Let $0 < E_0 < \infty$ be such that every energy strictly less than $E_0$ is good. Let $\phi^{(n)}[0] \in \S$ be a sequence of initial data with the energy bound \eqref{enbound}, which is uniformly localised in frequency in the sense of \eqref{spacloc}.  Suppose also that there exists $\eps > 0$ (independent of $n$) and $R_n \to \infty$ such that
$$ \int_{|x| \leq 1} \T_{00}(\phi^{(n)})(0,x)\ dx \geq \eps$$
and
$$ \int_{|x| \geq R_n} \T_{00}(\phi^{(n)})(0,x)\ dx \geq \eps.$$
Then there exists $A, M > 0$ independent of $n$ such that for each sufficiently large $n$ and every compact time interval $I \ni 0$, one can extend $\phi^{(n)}[0]$ to a classical wave map $(\phi^{(n)}, I)$ with $(A,1)$-entropy at most $M$.
\end{theorem}

This theorem is analogous to \cite[Proposition 4.7]{gopher}.  That proposition was proven by a strategy similar to that used to prove the corresponding result in frequency space (\cite[Proposition 4.3]{gopher}).  Namely, one partitions the initial data into a ``local'' component near the spatial origin, and a ``global'' component far away from the spatial origin, plus an error term corresponding to the intermediate region which one can take to be negligible using the pigeonhole principle.  The local and global components have strictly smaller energy than $E_0$, and can thus be evolved separately to $I$ by the hypotheses on $E_0$.  For small and medium times, one can use (approximate) finite speed of propagation to keep the local and global components separated from each other in space.  For large times, finite speed of propagation is no longer useful, but the pseudoconformal conservation law for NLS was used in \cite{gopher} to show that the local component became dispersed at large times, and thus had a negligible interaction with the global component.  Superimposing the two components and using perturbation theory gives the claim.

In \cite{tao:heatwave5} we adapt these arguments to wave maps.  Once again we need to scalarise the flow using the resolution $\phi \mapsto \psi_s$.  One can use exact finite speed of propagation for wave maps as a substitute for approximate finite speed of propagation (this being one of the rare places where the wave maps arguments are in fact slightly simpler than their NLS counterparts).  For the late times, the pseudoconformal identity argument is not available, but one can insead use the linear theory, finite speed of propagation, and the bounded entropy of the local solution to show that that solution must disperse after a bounded amount of time.

\subsection{Main result}

We can now state the main result of this paper:

\begin{theorem}[Reduction of main conjecture]\label{mainthm}  Suppose there is a notion of an $(A,\mu)$-wave map for which 
Theorems \ref{symscat}, \ref{freqbound}, \ref{spacbound}, \ref{spacdeloc} hold.  Then Conjecture \ref{minimal} (and hence Conjecture \ref{conj2} and Conjecture \ref{conj}) holds.
\end{theorem}

The proof of this theorem beigns in Section \ref{overview-sec}.

Thus in order to conclude the proof of the global regularity conjecture for wave maps into hyperbolic space, one must construct a notion of an $(A,\mu)$-wave map for which Theorems \ref{symscat}, \ref{freqbound}, \ref{spacbound}, \ref{spacdeloc} are true.  This is the purpose of the sequel \cite{tao:heatwave5} to this paper.

\subsection{Acknowledgements}

This project was started in 2001, while the author was a Clay Prize Fellow.  The author thanks Andrew Hassell and the Australian National University for their hospitality when a substantial portion of this work was initially conducted, and to Ben Andrews and Andrew Hassell for a crash course in Riemannian geometry and manifold embedding, and in particular to Ben Andrews for explaining the harmonic map heat flow.  The author also thanks Mark Keel for background material on wave maps, Daniel Tataru for sharing some valuable insights on multilinear estimates and function spaces, and to Igor Rodnianski and Jacob Sterbenz for valuable discussions.  The author is supported by NSF grant DMS-0649473 and a grant from the Macarthur Foundation.

\section{Notation}\label{notation-sec}

The dimension $m$ of the target hyperbolic space $\H^m$ is fixed throughout the paper, and all implied constants can depend on $m$.

We use $X = O(Y)$ or $X \lesssim Y$ to denote the estimate $|X| \leq CY$ for some absolute constant $C > 0$, that can depend on the $\delta_i$ and the dimension $m$ of the target hyperbolic space.  If we wish to permit $C$ to depend on some further parameters, we shall denote this by subscripts, e.g. $X = O_k(Y)$ or $X \lesssim_k Y$ denotes the estimate $|X| \leq C_k Y$ where $C_k > 0$ depends on $k$.  

Note that parameters can be other mathematical objects than numbers.  For instance, the statement that a function $u: \R^2 \to \R$ is Schwartz is equivalent to the assertion that one has a bound of the form $|\partial_x^k u(x)| \lesssim_{j,k,u} \langle x \rangle^{-j}$ for all $j,k \geq 0$ and $x \in \R^2$, where $\langle x \rangle := (1+|x|^2)^{1/2}$.

\section{Overview of argument}\label{overview-sec}

In this section we set out the ``high-level'' proof of Theorem \ref{mainthm}, breaking it up into a number of largely independent (and simpler) components which we will then prove separately.  

Recall that an energy $E > 0$ is \emph{good} if there exists $A,M > 0$, $0 < \mu \leq 1$ such that every classical wave map $(\phi,I)$ with $I$ compact and $\E(\phi) \leq E$ has an $(A,\mu)$-entropy of at most $M$.  Call $E$ \emph{bad} if it is not good. 

Clearly if $E$ is good, then $E'$ is good for every $0 < E' < E$.  From Theorem \ref{symscat}(v) we see that all sufficiently small energies $E$ are good.

\begin{lemma}[Scattering bound implies global well-posedness]\label{quant} Suppose $E$ is good.  Then Conjecture \ref{conj2} holds for all maximal Cauchy developments with energy less than $E$.  
\end{lemma}

\begin{proof} 
We first claim that any classical maximal Cauchy development  $\phi: I \to \Energy$ of energy less than $E$ is necessarily global.  Suppose this is not the case; by time reversal symmetry we may assume that $t_+ := \sup(I)$ is finite.  

Let $K$ be any compact subinterval of $I$. As $E$ is good, one can cover $K$ by $O_{E,\mu}(1)$ intervals $K_1,\ldots,K_m$ for some $0 < \mu \leq 1$, such that $(\phi\downharpoonright_{K_j},K_j)$ is $(O_E(1),\mu)$-wave map for each $1 \leq j \leq m$.  By Theorem \ref{symscat}(iii), we may assume that $\mu$ is small depending on $E$. Letting $K$ increase to $I$, we conclude the existence of some $t_* < t_+$ such that the restriction of $\phi$ to $[t_*,t_+-\eps]$ is a $(O_E(1),\mu)$-wave map for all $\eps > 0$ (otherwise one could show inductively that all the $K_i$ would have their endpoints converging to $t_+$ as $K$ approached $I$, a contradiction).  By Theorem \ref{symscat}(iv), $\phi$ can now be continuously extended in $\Energy$ to $t_+$, contradicting Theorem \ref{lwp-claim}(v).

Now we establish that any \emph{energy class} maximal Cauchy development  $\phi: I \to \Energy$ of energy less than $E$ is necessarily global as well.
Again, we suppose for contradiction that $t_+ := \sup(I)$ is finite.

Let $t_0$ be a point in the interior of $I$.  By Theorem \ref{energy-claim}, we can find a sequence of classical data $(\phi^{(n)}_0, \phi^{(n)}_1) \in \S$ such that $\iota(\phi^{(n)}_0, \phi^{(n)}_1)$ converges in $\Energy$ to $\phi[t_0]$, and has energy less than $E$.  By the preceding arguments, we know that each such classical data extends to a global classical wave map $(\phi^{(n)},\R)$.  From Theorem \ref{lwp-claim}(iv), $\iota(\phi^{(n)})$ converges locally uniformly on compact subintervals of $I$ to $\phi$. 

Arguing as before, we may find $t_* < t_+$ and a sequence $\eps_n \to 0$ such that the restriction of $\phi^{(n)}$ to $[t_*,t_+-\eps_n]$ is a $(O_E(1),\mu)$-wave map for infinitely many $n$.  Applying Theorem \ref{symscat}(iv) again, we can extend $\phi$ continously to $t_+$, again contradicting Theorem \ref{lwp-claim}(v).
\end{proof}

Now suppose that Conjecture \ref{conj2} fails.  Then by Lemma \ref{quant}, $E$ must be bad for at least one $E > 0$.  We conclude that there exists a unique \emph{minimal blowup energy} $0 < E_\crit < +\infty$ such that $E$ is bad for all $E > E_\crit$, and $E$ is good for all $E < E_\crit$.

By the definition of $E_\crit$ (and Theorem \ref{symscat}(iii)), we may now find a sequence $(\phi^{(n)}, I_n)$ of classical wave maps on compact time intervals $I_n$ such that 
\begin{equation}\label{ecr}
\E(\phi^{(n)}) \to E_\crit,
\end{equation}
and such the $(A_n,1)$-entropy (say) of $(\phi^{(n)},I_n)$ goes to infinity as $n \to \infty$, where $A_n \to \infty$.  The strategy will be to try to extract some sort of limit from some subsequence of (suitable renormalisations of) the $\phi^{(n)}$ to obtain the desired almost periodic maximal Cauchy development $\phi: I \to \Energy$ with non-zero energy (indeed, we will show that it has energy $E_\crit$).  In order to do this, we need to to ensure that the $\phi^{(n)}$ are compact (modulo symmetries) in a sufficiently strong sense.

Because the energy and notion of $(A,\mu)$-wave maps are both invariant under the four symmetries \eqref{space-trans}, \eqref{time-trans}, \eqref{time-reverse}, \eqref{cov-scaling}, we have the freedom to apply these transformations to each of the $\phi^{(n)}$ separately as we please.  We will frequently take advantage of this freedom in the sequel to achieve various normalisations.

The first stage in the compactification procedure is to ensure that the energy distribution of $\phi^{(n)}$ is localised in frequency.  To do this, we will introduce the notion of the \emph{energy spectral density} $\ESD(\phi_0,\phi_1): \R^+ \to \R^+$ of a given classical pair of data $(\phi_0,\phi_1) \in \S$.  The precise definition of this density, which involves the harmonic map heat flow, will be given in Section \ref{esd-sec}.  In that section we will also we establish the following basic properties of this energy spectral density function:

\begin{proposition}[Basic properties of ESD]\label{esd-prop}  Let $(\phi_0,\phi_1) \in \S$, and let $\ESD(\phi_0,\phi_1)$ be the quantity defined in Section \ref{esd-def}.
\begin{itemize}
\item[(i)] (Energy identity) $\ESD(\phi_0,\phi_1)(s)$ is continuous in $s$, and we have
\begin{equation}\label{energy-ident-0}
\E(\phi_0,\phi_1) = \int_0^\infty \ESD(\phi_0,\phi_1)(s)\ ds.
\end{equation}
\item[(ii)] (Symmetries) We have
\begin{align}
\ESD(\Trans_{x_0}(\phi_0,\phi_1))(s) &= \ESD(\phi_0,\phi_1)(s) \label{spacetrans-esd} \\
\ESD(\Rot_U(\phi_0,\phi_1))(s) &= \ESD(\phi_0,\phi_1)(s) \label{rot-esd} \\
\ESD(\Rev(\phi_0,\phi_1))(s) &= \ESD(\phi_0,\phi_1)(s) \label{timerev-esd} \\
\ESD(\Dil_{\lambda}(\phi_0,\phi_1))(s) &= \lambda^{-2} \ESD(\phi_0,\phi_1)(s/\lambda^2) \label{dil-esd} 
\end{align}
for $x_0 \in \R^2$, $U \in SO(m,1)$ and $\lambda > 0$.
\item[(iii)] (Convergence) If $(\phi^{(n)}_0,\phi^{(n)}_1) \in \S$ is such that $\iota(\phi^{(n)}_0,\phi^{(n)}_1)$ is a Cauchy sequence in $\Energy$, then on any compact interval $S \subset (0,+\infty)$, $\ESD( \phi^{(n)}_0, \phi^{(n)}_1 )$ is a Cauchy sequence in the uniform topology.  In particular, $\ESD(\Phi)$ can be meaningfully defined for $\Phi \in \Energy$.
\end{itemize}
\end{proposition}

From \eqref{energy-ident-0} we see in particular that
\begin{equation}\label{esdloc}
\int_0^\infty \ESD(\phi^{(n)}[t])(s)\ ds = E_\crit + o_{n \to \infty}(1)
\end{equation}
for all $n$ and all $t \in I_n$.

\subsection{Frequency localisation}

We now establish

\begin{proposition}[Frequency localisation]\label{freqloc}  For every $n$ and every $t \in I_n$ there exists $s_n(t) \in \R^+$, such that for every $\eps > 0$ there exists a quantity $C(\eps) > 0$, such that
$$ \int_{s < s_n(t)/C(\eps)} \ESD(\phi^{(n)}[t])(s)\ ds + \int_{s > C(\eps) s_n(t)} \ESD(\phi^{(n)}[t])(s)\ ds \leq \eps$$
for all sufficiently large $n$ and all $t \in I_n$.
\end{proposition}

To prove this proposition, we first observe that it suffices to establish an apparently weaker version in which $s_n(t)$ is allowed to depend on $\eps$:

\begin{proposition}[Frequency localisation, II]\label{freqloc2}  For every $\eps > 0$ there exists $C(\eps) > 0$, such that every sufficiently large $n$ and every $t \in I_n$ there exists $s_{n,\eps}(t) \in \R^+$, such that 
\begin{equation}\label{jon}
 \int_{s < s_{n,\eps}(t)/C(\eps)} \ESD(\phi^{(n)}[t])(s)\ ds + \int_{s > C(\eps) s_{n,\eps}(t)} \ESD(\phi^{(n)}[t])(s)\ ds \leq \eps.
 \end{equation}
\end{proposition}

Indeed, suppose that Proposition \ref{freqloc2} was true.  From \eqref{esdloc}, \eqref{jon} we see that given any sufficiently small $\eps > \eps' > 0$, one has 
$$\frac{s_{n,\eps'}(t)}{C(\eps) C(\eps')} \leq s_{n,\eps}(t) \leq C(\eps) C(\eps') s_{n,\eps'}(t)$$
for all sufficiently large $n$ and all $t \in I_n$.  If one then sets $s_n(t) := s_{n,\eps_0}(t)$ for some sufficiently small $\eps_0$ one obtains the claim.

It remains to prove Proposition \ref{freqloc2}.  We argue by contradiction.  If this proposition failed, then there would exist an $\eps > 0$ with the property that for any $C > 0$, there existed arbitrarily large $n$ and a time $t_n \in I_n$ such that
$$
 \int_{s < s_0/C} \ESD(\phi^{(n)}[t_n])(s)\ ds + \int_{s > C s_0} \ESD(\phi^{(n)}[t_n])(s)\ ds \geq \eps$$
for every $s_0 > 0$.
 
Fix this $\eps$, which we may assume to be small.  We now allow all implied constants to depend on $\eps$.

By passing to a diagonal subsequence if necessary, one can find $C_n \to \infty$ and $t_n \in I_n$ such that
$$
 \int_{s < s_0/C_n} \ESD(\phi^{(n)}[t_n])(s)\ ds + \int_{s > C s_0} \ESD(\phi^{(n)}[t_n])(s)\ ds \geq \eps$$
for every $s_0 > 0$.

By the intermediate value theorem, we can find $s_n$ to be such that
$$  \int_{s < s_n/C_n} \ESD(\phi^{(n)}[t_n])(s)\ ds = \eps/2$$
and thus
$$  \int_{s > C_n s_n} \ESD(\phi^{(n)}[t_n])(s)\ ds \geq \eps/2.$$

By the pigeonhole principle, one can find $s'_n$ between $s_n/\sqrt{C_n}$ and $\sqrt{C_n} s_n$ such that
$$ \int_{K_n s'_n \leq s \leq s'_n/K_n} \ESD(\phi^{(n)}[t_n])(s)\ ds  \leq 1/K_n^{100}$$
for some $K_n \to \infty$ and sufficently large $n$ (e.g. one could take $K_n = \log \log C_n$).  

By using time translation and dilation symmetry, we may make $t_n=0 \in I_n$ and $s'_n=1$, thus we have
\begin{equation}\label{jim}
 \int_{1/K_n < s < K_n} \ESD(\phi^{(n)}[0])(s)\ ds = 1/K_n^{100}
\end{equation}
and
$$  \int_{s \geq K_n} \ESD(\phi^{(n)}[0])(s)\ ds, \int_{s \leq 1/K_n} \ESD(\phi^{(n)}[0])(s)\ ds \geq \eps/4.$$
for sufficiently large $n$.  But then by applying Theorem \ref{freqbound} (and the uniqueness theory for classical wave maps) we conclude that $(\phi^{(n)},I_n)$ has a $(O(1),1)$-entropy of $O(1)$ for sufficiently large $n$, contradicting the hypothesis that the $(A_n,1)$-entropy of this wave map goes to infinity (thanks to Theorem \ref{symscat}(i)).  This concludes the proof of Proposition \ref{freqloc2} and hence Proposition \ref{freqloc}.

\subsection{Spatial concentration}

For the next step, we borrow a trick from \cite{borg:scatter}, \cite{borg:book}, and work with the \emph{middle third} of the inteval $I_n$, though for minor technical reasons it is more convenient to work with a middle half.  From the greedy algorithm (and Theorem \ref{symscat}(i)), we see that we can partition each $I_n$ into four consecutive time compact intervals $I_n = I_n^{--} \cup I_n^- \cup I_n^+ \cup I_n^{++}$ such that $(\phi^{(n)}\downharpoonright_{I_n^\alpha},I_n^\alpha)$ has $(A_n,1)$-entropy going to infinity as $n \to \infty$ for $\alpha = --, -, +, ++$.

In the middle half $I_n^- \cup I_n^+$, we will be able to achieve an analogue of Proposition \ref{freqloc} in physical space.  Define the \emph{energy density} $\T_{00}(\phi_0,\phi_1)(x)$ of data $(\phi_0,\phi_1)$ by the formula
$$ \T_{00}(\phi_0,\phi_1)(x) := 
\frac{1}{2} |\phi_1|_{h(\phi_0)}^2
+ \frac{1}{2} |\nabla_x \phi_0|_{h(\phi_0)}^2,$$
thus
$$ \E(\phi_0,\phi_1) = \int_{\R^2} \T_{00}(\phi_0,\phi_1)(x)\ dx.$$

We now use Theorem \ref{spacbound} to conclude

\begin{proposition}[Spatial concentration in the middle half]\label{spatloc-0}  There exists $\eps_0 > 0$ such that for 
 every $n$ and every $t \in I_n^- \cup I_n^+$ there exists $x_n(t) \in \R^2$ such that
$$ \int_{|x-x_n(t)| \leq s_n(t)^2} \T_{00}(\phi^{(n)}[t])(x) dx \geq \eps_0.$$
\end{proposition}

\begin{proof}  Suppose this were not the case, then after passing to a subsequence one could find $t_n \in I_n^- \cup I_n^+$ such that
$$ \sup_{x_0 \in \R^2} \int_{|x-x_0| \leq s_n(t_n)^2} \T_{00}(\phi^{(n)}[t_n])(x) dx = o_{n \to \infty}(1).$$
By time translation and dilation symmetries \eqref{time-trans}, \eqref{cov-scaling} we may normalise $t_n=0$ and $s_n(t_n)=1$, thus
$$ \sup_{x_0 \in \R^2} \int_{|x-x_0| \leq 1} \T_{00}(\phi^{(n)}[0])(x) dx = o_{n \to \infty}(1).$$
But then by Theorem \ref{spacbound} (and the uniqueness theory for classical wave maps, and Theorem \ref{symscat}(i)) we see that either $(\phi^{(n)}\downharpoonright_{I_n^{--}}, I_n^{--})$ or $(\phi^{(n)}\downharpoonright_{I_n^{++}}, I_n^{++})$ has an $(O(1),1)$-entropy of $O(1)$ for sufficiently large $n$, a contradiction.
\end{proof}

Similarly, we can use Theorem \ref{spacdeloc} to conclude

\begin{proposition}[Spatial localisation in the middle half]\label{spatloc}  Let the notation and assumptions be as above.  Then for 
 every $n$ and every $t \in I_n^- \cup I_n^+$ there exists $x_n(t) \in \R^2$, such that for every $\eps > 0$ there exists a quantity $C(\eps) > 0$, such that
$$ \int_{|x-x_n(t)| \geq C(\eps) s_n(t)^2} \T_{00}(\phi^{(n)}[t])(x) dx \leq \eps.$$
\end{proposition}

\begin{proof}  Let $x_n(t)$ be as in Proposition \ref{spatloc-0}.  Suppose the claim failed for this choice of $x_n(t)$; then one can find an $\eps > 0$ such that for any $C > 0$ one can find arbitrarily large $n$ such that
$$ \int_{|x-x_n(t_n)| \geq C s_n(t_n)^2} \T_{00}(\phi^{(n)}[t_n])(x) dx \geq \eps.$$
In particular, by passing to a subsequence, one can find $C_n \to \infty$ such that
$$ \int_{|x-x_n(t_n)| \geq C_n s_n(t_n)^2} \T_{00}(\phi^{(n)}[t_n])(x) dx \geq \eps.$$
By time translation and dilation symmetries \eqref{time-trans}, \eqref{cov-scaling} we may normalise $t_n=0$ and $s_n(t_n)=1$, thus
$$ \int_{|x-x_n(0)| \geq C_n} \T_{00}(\phi^{(n)}[0])(x) dx \geq \eps.$$
Meanwhile, from Proposition \ref{spatloc-0}
$$ \int_{|x-x_n(0)| \leq 1} \T_{00}(\phi^{(n)}[0])(x) dx \geq \eps_0.$$
Applying Theorem \ref{spacdeloc} (and the uniqueness theory for classical wave maps), we conclude that $(\phi^{(n)},I_n)$ has a $(O(1),1)$-entropy of $O(1)$ for sufficiently large $n$, a contradiction.
\end{proof}

Now we use the symmetries \eqref{space-trans}, \eqref{time-trans}, \eqref{cov-scaling} to normalise the solutions $\phi^{(n)}$.  First, by using \eqref{time-trans} (and Theorem \ref{symscat}), we may assume that the intervals $I_n^-, I_n^+$ meet at the time origin $t=0$ for all $n$.  Next, by the dilation symmetry \eqref{cov-scaling} (and \eqref{dil-esd}), we may assume that $s_n(0) = 1$ for all $n$.  Finally, by using the translation symmetry \eqref{space-trans} (and \eqref{spacetrans-esd}), we may assume that $x_n(0)=0$ for all $n$.  In particular, we now see that for every $\eps > 0$ we have $C(\eps) > 0$ such that
$$ \int_{s < 1/C(\eps)} \ESD(\phi^{(n)}[0])(s)\ ds + \int_{s > C(\eps)} \ESD(\phi^{(n)}[0])(s)\ ds \leq \eps$$
and
$$ \int_{|x| \geq C(\eps)} \T_{00}(\phi^{(n)}[0])(x) dx \leq \eps$$
for all sufficiently large $n$.

\subsection{Conclusion of the argument}

To take advantage of the above localisations we use the following proposition, proven in Section \ref{compact-sec}:

\begin{proposition}[Compactness]\label{precom}  Let $C: \R^+ \to \R^+$ be a function and $E \geq 0$.  Then there exists a compact subset $K$ of $\Energy$ with the following properties:
\begin{itemize}
\item[(i)] If $(\phi^{(n)}_0, \phi^{(n)}_1) \in \S$ is a sequence with $\E(\phi^{(n)}_0, \phi^{(n)}_1) \leq E$ with the property that for each $\eps > 0$, one has
\begin{equation}\label{esd-n}
\int_{s < 1/C(\eps)} \ESD(\phi^{(n)}_0,\phi^{(n)}_1)(s)\ ds + \int_{s > C(\eps)} \ESD(\phi^{(n)}_0,\phi^{(n)}_1)(s)\ ds \leq \eps
\end{equation}
and
\begin{equation}\label{locale}
 \int_{|x| \geq C(\eps)} \T_{00}(\phi^{(n)}_0,\phi^{(n)}_1)(x) dx \leq \eps
\end{equation}
for all sufficiently large $n$, then there is a subsequence $(\phi^{(n_j)}_0, \phi^{(n_j)}_1)$ such that $\iota(\phi^{(n_j)}_0, \phi^{(n_j)}_1)$ converges in $\Energy$ to a limit in $K$.
\item[(ii)]  For all $\Phi \in K$, one has
$$ \int_{s < 1/C(\eps)} \ESD(\Phi)(s)\ ds + \int_{s > C(\eps)} \ESD(\Phi)(s)\ ds \leq \eps$$
and
$$ \int_{|x| \geq C(\eps)} \T_{00}(\Phi)(x) dx \leq \eps$$
for all $\eps > 0$.
\end{itemize}
\end{proposition}

Using this proposition, we see (after passing to a subsequence) that $\iota(\phi^{(n)}[0])$ converges in $\Energy$ to some limit $\Phi_0 \in \Energy$.  Since $\E(\phi^{(n)}[0]) \to E_{\crit}$, we conclude (by Theorem \ref{energy-claim}(iv)) that
$$ \E( \Phi_0 ) = E_\crit.$$
In particular, $\Phi_0$ has non-zero energy.

Now let $\phi: I \to \Energy$ be the maximal Cauchy development with $\phi[0] = \Phi_0$, as given by Theorem \ref{lwp-claim}.  For any compact subinterval $J$ of $I$ containing $0$, we claim that $J \subset I_n^- \cup I_n^+$ for all sufficiently large $n$.  Indeed, from Corollary \ref{finsiz} we see that the $(A_n,1)$-entropy of $(\phi^{(n)}\downharpoonright_J, J)$ is bounded in $n$ for sufficiently large $n$, while from construction the $(A,1)$ entropy of $(\phi^{(n)}\downharpoonright_{I_n^-}, I_n^-)$, $(\phi^{(n)}\downharpoonright_{I_n^+}, I_n^+)$ go to infinity.  By Theorem \ref{symscat}(i) we conclude that $J$ cannot contain either $I_n^-$ or $I_n^+$, and the claim follows.

Now let $t \in I$, then $\E(\phi[t]) = E_\crit$, and $t \in I_n^- \cup I_n^+$ for all sufficiently large $n$.  From Theorem \ref{lwp-claim}(iv), we see that $\iota(\phi^{(n)}[t])$ converges in $\Energy$ to $\phi[t]$.  On the other hand, for each $0 < \eps < E_\crit$ we have
$$ \int_{s < s_n(t)/C(\eps)} \ESD(\phi^{(n)}[t])(s)\ ds + \int_{s > C(\eps) s_n(t)} \ESD(\phi^{(n)}[t])(s)\ ds \leq \eps$$
for all sufficiently large $n$.   If some subsequence of $s_n(t)$ converges to $0$ or $+\infty$ then we see from Proposition \ref{esd-prop} that $\E(\phi[t]) < E_\crit$, contradiction.  So $s_n(t)$ must instead have an accumulation point at some $0 < s(t) < \infty$.  From Proposition \ref{esd-prop}(iii) and Fatou's lemma, we conclude that
$$ \int_{s < s(t)/C(\eps)} \ESD(\phi[t])(s)\ ds + \int_{s > C(\eps) s(t)} \ESD(\phi[t])(s)\ ds \leq \eps.$$
Next, we have
$$ \int_{|x-x_n(t)| \geq C(\eps) s_n(t)^2} \T_{00}(\phi^{(n)}[t])(x) dx \leq \eps.$$
Let us temporarily restrict to a subsequence where $s_n(t) \to s(t)$.  From Theorem \ref{energy-claim}(iv) we see as before that $x_n(t)$ cannot have a subsequence diverging to infinity, and instead has a subsequence converging to some $x(t) \in \R^2$; taking limits again we conclude
$$ \int_{|x-x(t)| \geq C(\eps) s(t)^2} \T_{00}(\phi[t])(x) dx \leq \eps.$$
Applying Proposition \ref{precom}, we conclude that
$$ \Dil_{s(t)^{-1/2}} \Trans_{-x(t)} \phi[t] \in K$$
for some compact subset $K$ of $\Energy$, and Theorem \ref{mainthm} follows.

\section{The energy space}\label{engen}

In this section we recall the precise definition of the energy space defined in \cite{tao:heatwave2}.  We will also define the energy spectral distribution $\ESD(\phi_0,\phi_1)$ used in the proof of Theorem \ref{minimal}, and establish its basic properties in Proposition \ref{esd-prop}.  Finally, we will also establish key compactness property, see Proposition \ref{precom}.

\subsection{The caloric gauge}

We first recall a global existence theorem for the harmonic map heat flow, essentially due to Eells and Sampson \cite{eells}.

\begin{theorem}[Global existence for heat flow]\label{ghp}  Let $I$ be a compact interval, and let $\phi: I \times \R^2 \to \H$ be a smooth map differs from a constant $\phi(\infty)$ by a Schwartz function.  Then there exists a unique extension $\phi: \R^+ \times I \times \R^2 \to H$ of $\phi$ (thus $\phi(0,t,x) = \phi(t,x)$ for all $(t,x) \in I \times \R^2$) which is smooth with all derivatives bounded, and obeys the harmonic map heat flow equation
$$ \partial_s \phi = (\phi^* \nabla)_i \partial_i \phi$$
on $\R^+ \times I \times \R^2 := \{ (s,t,x): s \in \R^+, t \in I, x \in \R^2\}$, and converges in $C^\infty_{\loc}(I \times \R^2)$ to $\phi(\infty)$ as $s \to \infty$.
\end{theorem}

\begin{proof} See \cite[Theorem 3.16]{tao:heatwave2}.
\end{proof}

This result allows us to extend any classical wave map $\phi \in WM(I,E)$ by harmonic map heat flow into the region $\R^+ \times I \times \R^2$.  In order to analyse this heat flow we will place an orthonormal frame on this map (or more precisely on the pullback tangent bundle $\phi^* T\H$).

Given any point $p \in \H$, define an \emph{orthonormal frame} at $p$ to be any orthogonal orientation-preserving map $e: \R^m \to T_p \H$ from $\R^m$ to the tangent space at $p$ (with the metric $h(p)$, of course), and let $\Frame(T_p \H)$ denote the space of such frames; note that this space has an obvious transitive action of the special orthogonal group $SO(m)$.  We then define the \emph{orthonormal frame bundle} $\Frame( \phi^* T\H)$ of $\phi$ to be the space of all pairs $((s,t,x), e)$ where $(s,x) \in \R^+ \times I \times \times \R^2$ and $e \in \Frame(T_{\phi(s,x)}\H)$; this is a smooth vector bundle over $\R^+ \times I \times \R^2$.  We then define an \emph{orthonormal frame} $e \in \Gamma(\Frame(\phi^* T\H))$ for $\phi$ to be a section of this bundle, i.e. a smooth assignment $e(s,x) \in \Frame(T_{\phi(s,x)}\H)$ of an orthonormal frame at $\phi(s,x)$ to every point $(s,x) \in \R^+ \times I \times \R^2$.

Each orthonormal frame $e \in \Gamma(\Frame(\phi^* T\H))$ provides an orthogonal, orientation-preserving identification between the vector bundle $\phi^* T \H$ (with the metric $\phi^* h$) and the trivial bundle $(\R^+ \times I \times \R^2) \times \R^m$ (with the Euclidean metric on $\R^m$), thus sections $\Psi \in \Gamma(\phi^* T\H)$ can be pulled back to functions $e^* \Psi: \R^+ \times I \times \R^2 \to \R^m$ by the formula $e^* \Psi := e^{-1} \circ \Psi$.  The connection $\phi^* \nabla$ on $\phi^* T\H$ can similarly be pulled back to a connection $D$ on the trivial bundle $(\R^+ \times \R^2) \times \R^m$, defined by
\begin{equation}\label{D-def}
 D_i := \partial_i + A_i
 \end{equation}
where $A_i \in \mathfrak{so}(m)$ is the skew-adjoint $m \times m$ matrix field is given by the formula
\begin{equation}\label{Adef}
(A_i)_{ab} = \langle (\phi^* \nabla)_i e_a, e_b \rangle_{\phi^* h}
\end{equation}
where $e_1,\ldots,e_m$ are the images of the standard orthonormal basis for $\R^m$ under $e$.  Of course one similarly has covariant derivatives $D_t = \partial_t + A_t$,  $D_s = \partial_s + A_s$ in the $t$ and $s$ directions.  

Given such a frame, we define the \emph{derivative fields}
$\psi_j: \R^+ \times I \times \R^2 \to \R^m$ by the formula 
\begin{equation}\label{psij-def}
\psi_j := e^* \partial_j \phi,
\end{equation}
and similarly define 
\begin{equation}\label{psis-def}
\psi_s := e^* \partial_s \phi; \quad \psi_t := e^* \partial_t \phi.
\end{equation}
We record the zero-torsion property
\begin{equation}\label{zerotor-frame}
D_i \psi_j = D_j \psi_i
\end{equation}
and the constant negative curvature property
\begin{equation}\label{curv-frame}
[D_i, D_j] = \partial_i A_j - \partial_j A_i + [A_i,A_j] = - \psi_i \wedge \psi_j
\end{equation}
where $\psi_i \wedge \psi_j$ is the anti-symmetric matrix field
$$ \psi_i \wedge \psi_j := \psi_i \psi_j^* - \psi_j \psi_i^*.$$
The harmonic map heat flow equation becomes
\begin{equation}\label{heatflow}
\psi_s = D_i \psi_i = \partial_i \psi_i + A_i \psi_i.
\end{equation}
The wave maps equation \eqref{cov} becomes
\begin{equation}\label{w-vanish}
w|_{s=0} = 0
\end{equation}
where $w: \R^+ \times I \times \R^2 \to \R^m$ is the \emph{wave-tension field}
\begin{equation}\label{w-def}
w := D^\alpha \psi_\alpha = \partial^\alpha \psi_\alpha + A^\alpha \psi_\alpha.
\end{equation}

Following \cite{tao:heatwave2}, we say that an orthonormal frame $e$ is a \emph{caloric gauge} if one has 
\begin{equation}\label{ass}
A_s = 0
\end{equation}
throughout $\R^+ \times I \times \R^2$, and if we have 
\begin{equation}\label{esx}
\lim_{s \to \infty} e(s,t,x) = e(\infty)
\end{equation}
for all $x \in \R^2$ and some constant $e(\infty) \in \Frame(T_{\phi(\infty)}\H)$.  

We have the following existence theorem for this gauge:

\begin{theorem}[Existence of caloric gauge]\label{dynamic-caloric}  Let $I$ be a time interval with non-empty interior, let $\phi: I \times \R^2 \to \H$ be a smooth map which differs from a constant $\phi(\infty)$ by a Schwartz function in space, and let $e(\infty) \in \Frame(T_{\phi(\infty)}\H)$ be a frame for $\phi(\infty)$.  Let $\phi: \R^+ \times I \times \R^2 \to \H$ be the heat flow extension from Theorem \ref{ghp}.  Then there exists a unique smooth frame $e \in \Gamma(\Frame(\phi^* T\H))$ such that $e(t)$ is a caloric gauge for $\phi(t)$ which equals $e(\infty)$ at infinity for each $t \in I$.   All derivatives of $\phi - \phi(\infty)$ in the variables $t,x,s$ are Schwartz in $x$ for each fixed $t,s$.  In particular, these derivatives are uniformly Schwartz for $t, s$ in a compact range.
\end{theorem}

\begin{proof} See \cite[Theorem 3.16]{tao:heatwave2}.
\end{proof}

In the caloric gauge we have the derivative fields $\psi_t, \psi_x, \psi_s$ and the connection fields $A_t$, $A_x$ (recall in the caloric gauge that the field $A_s$ is trivial).  It will be convenient to introduce the combined vector or tensor fields
\begin{align*}
\Psi_x &:= (\psi_x, A_x) \\
\psi_{t,x} &:= (\psi_t, \psi_x) \\
A_{t,x} &:= (A_t, A_x) \\
\Psi_{t,x} &:= (\psi_{t,x}, A_{t,x}) \\
\Psi_{s,t,x} &:= (\psi_s, \psi_{t,x}, A_{t,x}).
\end{align*}
We refer to $\Psi_{s,t,x}$ as the \emph{differentiated fields} of $\phi$ in the caloric gauge associated with $e(\infty)$.  We also introduce the \emph{wave-tension field} $w := D^\alpha \psi_\alpha$; the wave map equation \eqref{cov} is then equivalent to the vanishing of this quantity on the $s=0$ boundary of $\R^+ \times I \times \R^2$.  We recall the basic equations of motion for these fields:

\begin{proposition}[Equations of motion]\label{abound}  Let $I$ be an interval, let $\phi \in \WM(I,E)$ for some $E > 0$, let $e$ be a caloric gauge for $\phi$, and let $\phi_t, \psi_x, \psi_s, A_t, A_x, w$ be the differentiated fields, connection fields, and wave-tension field.  Then we have the equations of motion
\begin{align}
A_{t,x}(s,t,x) &= \int_s^\infty \psi_s \wedge \psi_{t,x}(s',t,x)\ ds' \label{a-eq}\\
\psi_{t,x}(s,t,x) &= -\int_s^\infty D_{t,x} \psi_s(s',t,x)\ ds' \label{psi-eq}\\
\partial_s \psi_s &= D_i D_i \psi_s - (\psi_s \wedge \psi_i) \psi_i \label{psis-eq}\\
\partial_s \psi_{t,x} &= D_i D_i \psi_{t,x} - (\psi_{t,x} \wedge \psi_i) \psi_i \label{psit-eq}\\
\partial_s w &= D_i D_i w - (w \wedge \psi_i) \psi_i + 2 (\psi_\alpha \wedge \psi_i) D_i \psi^\alpha	 \label{w-eq} \\
D^\alpha D_\alpha \psi_s &= \partial_s w - (\psi_\alpha \wedge \psi_s) \psi^\alpha.   \label{psis-box}
\end{align}
\end{proposition}

\begin{proof} See \cite[Proposition 5.4]{tao:heatwave3}.
\end{proof}

We recall some useful estimates in the caloric gauge:

\begin{proposition}\label{corbound}  With the assumptions and notation in Proposition \ref{abound}, we have
\begin{align}
\int_0^\infty s^{k-1} \| \partial_x^k \psi_{t,x} \|_{L^2_x(\R^2)}^2 ds &\lesssim_{E,k} 1 \label{l2-integ-ord} \\
\sup_{s > 0} s^{(k-1)/2} \| \partial_x^{k-1} \psi_{t,x} \|_{L^2_x(\R^2)} &\lesssim_{E,k} 1 
\label{l2-const-ord} \\
\int_0^\infty s^{k-1} \| \partial_x^{k-1} \psi_{t,x} \|_{L^\infty_x(\R^2)}^2\ ds &\lesssim_{E,k} 1\label{linfty-integ-ord} \\
\sup_{s > 0} s^{k/2} \| \partial_x^{k-1} \psi_{t,x} \|_{L^\infty_x(\R^2)} &\lesssim_{E,k} 1 
\label{linfty-const-ord}
\end{align}
for all $k \geq 1$.  Similar estimates hold if one replaces $\psi_{t,x}$ with $A_x$, $\partial_x \psi_{t,x}$ with $\psi_s$, $\partial_x^2$ with $\partial_s$, and/or $\partial_x$ with $D_x$.  
\end{proposition}

\begin{proof} See \cite[Corollary 4.6]{tao:heatwave2}, \cite[Lemma 4.8]{tao:heatwave2}, and \cite[Proposition 4.3]{tao:heatwave2}.
\end{proof}

We also recall some useful estimates for the covariant heat equation:

\begin{lemma}[Covariant parabolic regularity]\label{uheat-est}   Let the assumptions and notation be as in Proposition \ref{abound}.  Let $u(0): \R^2 \to \R^m$ be smooth and compactly supported.  Then there exists a unique smooth solution $u: \R^+ \times \R^2 \to \R^m$ to the covariant heat equation
\begin{equation}\label{uheat}
\partial_s u = D_i D_i u - (u \wedge \psi_i) \psi_i
\end{equation}
with initial data $u(0)$ such that $u(s)$ is Schwartz for each $s$.  Furthermore we have the pointwise estimate
\begin{equation}\label{upoint}
|u(s)| \leq e^{s\Delta} |u(0)|
\end{equation}
the energy inequality
\begin{equation}\label{energy-ineq}
\partial_s \int_{\R^2} |u(s,x)|^2\ dx \leq - 2 \int_{\R^2} |D_x u(s,x)|^2\ dx
\end{equation}
and the parabolic estimates
\begin{align}
\sup_{s > 0} s^{k/2} \| \partial_x^k u(s) \|_{L^2_x(\R^2)} \lesssim_{E,k} \|u(0)\|_{L^2_x(\R^2)} \label{u-l2-fixed}\\
\sup_{s > 0} s^{(k+1)/2} \| \partial_x^k u(s) \|_{L^\infty_x(\R^2)} \lesssim_{E,k} \|u(0)\|_{L^2_x(\R^2)} \label{u-linfty-fixed}\\
\int_0^\infty s^k \| \partial_x^{k+1} u(s) \|_{L^2_x(\R^2)}^2\ ds \lesssim_{E,k} \|u(0)\|_{L^2_x(\R^2)}^2 \label{u-l2-integ}\\
\int_0^\infty s^k \| \partial_x^k u(s) \|_{L^\infty_x(\R^2)}^2\ ds \lesssim_{E,k} \|u(0)\|_{L^2_x(\R^2)}^2 \label{u-linfty-integ}
\end{align}
for all $k \geq 0$.  We also have the variant estimate
\begin{equation}\label{ulp}
\int_0^\infty s^{-2/p} \| u(s) \|_{L^p_x(\R^2)}^2\ ds \lesssim_p \|u(0)\|_{L^2_x(\R^2)}^2 
\end{equation}
for all $2 < p \leq \infty$, and the mass diffusion identity
\begin{equation}\label{udie}
\partial_s |u|^2 = \Delta |u|^2 - 2 |D_x u|^2 - |u \wedge \psi_x|^2.
\end{equation}
\end{lemma}

\begin{proof} See \cite[Lemma 4.8]{tao:heatwave2}.
\end{proof}

\subsection{Working in the caloric gauge}

Let $\phi \in \WM(I,E)$ be a classical wave map on a compact time interval $I$, quotiented out by Lorentz rotations $SO(m,1)$.  By Theorem \ref{dynamic-caloric}, given any such wave map, and given any frame $e(\infty)$ at $e(\infty) \in \Frame(T_{\phi(\infty)}\H)$, one can construct a caloric gauge for $\phi$, which then generates the differentiated fields $\Psi_{s,t,x}$.  We let $\WMC(\phi,I)$ denote the collection of all the possible fields $\Psi_{s,t,x}$ that can arise for a fixed $\phi$ by choosing $e(\infty)$, and let $\WMC(I,E) = \bigcup_{\phi \in \WM(I,E)} \WMC(\phi,I)$ denote all the differentiated fields that can arise from some classical wave map $\phi \in \WM(I,E)$.  

Observe that if $U \in SO(m)$ is a rotation, then replacing $e(\infty)$ by $e(\infty) \circ U^{-1}$ corresponds to rotating the fields $\psi_{t,x}$ to $U \psi_{t,x}$, and similarly rotating $A_{t,x}$ to $U A_{t,x} U^{-1}$.  This gives an action of $SO(m)$ on $\WMC(I,E)$, whose orbits are precisely the sets $\WMC(\phi,I)$.  Thus each classical wave map $\phi \in \WM(I,E)$ gives rise to an $SO(m)$-orbit $\WMC(\phi,I)$ of differentiated fields.  Conversely, it is not hard to see (from the Picard uniqueness theorem) that each differentiated field $\Psi_{s,t,x} \in \WMC(I,E)$ arises from exactly one wave map $\phi \in \WM(I,E)$ (quotienting out by $SO(m,1)$ as usual).  So we have an identification 
$$\WM(I,E) \equiv SO(m)\backslash \WMC(I,E).$$

We can play similar games with the initial data space $\S$: if $\Phi = (\phi_0,\phi_1) \in \S$ is classical initial data, we can construct associated caloric gauges $\Psi_{s,t,x}$ on $\R^+ \times \R^2$ for each choice of frame $\e(\infty) \in \Frame(T_{\phi_0(\infty)} \H)$, for instance by extending $\Phi$ for a short amount of time and then using Theorem \ref{dynamic-caloric}, or by using \cite[Theorem 3.12]{tao:heatwave2}, \cite[Lemma 4.8]{tao:heatwave2}.  We let $\SC(\Phi)$ denote the collection of all such differentiated fields for a fixed $\Phi$, and $\SC := \bigcup_{\Phi \in \S} \SC(\Phi)$ denote the total collection of such fields. Once again, the orthogonal group $SO(m)$ acts on $\SC$, with orbits $\SC(\Phi)$, with the orbit determining $\Phi$ up to the action of $SO(m,1)$, so that we have an identification 
$$SO(m,1) \backslash \S \equiv SO(m) \backslash \SC.$$
Also observe that if $\Psi_{s,t,x} \in \WMC(\phi,I)$ for some $\phi \in \WM(I,E)$, then $\Psi_{s,t,x}(t_0) \in \SC(\phi[t_0])$ for all $t_0 \in I$.  Finally, the constant data $\const \in \S$ corresponds to the zero differentiated field $0 \in \SC$, thus $\SC(\const) = \{0\}$ (and similarly $\WMC(\const) = \{0\}$.

\subsection{Construction of the energy space}\label{energydef-sec}

We introduce the Littlewood-Paley space
\begin{equation}\label{ldef}
 {\mathcal L} := L^2( \R^+ \times \R^2 \to \R^m, dx ds ) \times L^2( \R^2 \to \R^m, \frac{1}{2} dx ).
\end{equation}
We can define an energy metric $d_{\Energy}$ on $\SC$ by the formula
\begin{equation}\label{psipsi}
\begin{split}
d_\Energy( \Psi_{s,t,x}, \Psi'_{s,t,x} ) &:= \| ( \psi_s - \psi'_s, \psi_t(0) - \psi'_t(0) ) \|_{\mathcal L} \\
&= \left(\int_0^\infty \| \delta \psi_s(s) \|_{L^2_x(\R^2)}^2\ ds + \frac{1}{2} \| \delta \psi_t(0) \|_{L^2_x(\R^2)}^2\right)^{1/2}.
\end{split}
\end{equation}
We also abbreviate $d_{\Energy}(\Psi_{s,t,x},0)^2$ as $\E(\Psi_{s,t,x})$, and refer to this as the \emph{energy} of the differentiated field $\Psi_{s,t,x}$.

Observe that this metric is preserved by the action of $SO(m)$.  It thus descends to a metric on $SO(m,1) \backslash \S \equiv SO(m) \backslash \SC$ in the usual manner; by abuse of notation we also denote this metric as $d_\Energy$, thus
$$
d_\Energy(\Phi, \Phi') := \inf_{\Psi_{s,t,x} \in \SC(\Phi)} \inf_{\Psi'_{s,t,x} \in \SC(\Phi')}
d_\Energy(\Psi_{s,t,x}, \Psi'_{s,t,x} ).
$$
We define the energy space\footnote{This definition is equivalent to that in \cite{tao:heatwave2}, though arranged slightly differently.} $\Energy$ to be the completion of $SO(m,1) \backslash \S$ using the metric $\Energy$, and $\iota$ to be the quotient map from $\S$ to $SO(m,1) \backslash \S$.    With these definitions, Theorem \ref{energy-claim} was established in \cite{tao:heatwave2}.  Observe from Theorem \ref{energy-claim}(iv) that the energy of a map $\Phi \in \S$ is equal to the energy of any of its differentiated fields $\Psi_{s,t,x} \in \SC(\Phi)$.

\subsection{The energy spectral distribution}\label{esd-sec}

Let $\Phi \in \S$ and $\Psi_{s,t,x} \in \SC(\Phi)$.  In \cite[Lemma 5.3]{tao:heatwave2} the energy identity
$$
\E(\Phi) = \int_0^\infty \| \psi_s(s) \|_{L^2_x(\R^2)}^2\ ds + \frac{1}{2} \| \psi_t(0) \|_{L^2_x(\R^2)}^2
$$
was established.  From Lemma \ref{uheat-est}, we see that $\psi_t(0)$ can be extended to a smooth map $\psi_t: \R^+ \times \R^2 \to \R^m$ obeying the covariant heat equation
\begin{equation}\label{covheat}
 \partial_s \psi_t = D_i D_i \psi_t - (\psi_t \wedge \psi_i) \psi_i 
\end{equation}
(i.e. the time component of \eqref{psit-eq})
obeying the pointwise bound
$$ |\psi_t(s)| \leq e^{s\Delta} |\psi_t(0)|$$
so in particular $\| \psi_t(s) \|_{L^2_x(\R^2)} \to 0$ as $s \to \infty$; from Lemma \ref{uheat-est} we also have the energy identity
\begin{equation}\label{psits}
 \partial_s \| \psi_t(s) \|_{L^2_x(\R^2)}^2 = - \int_{\R^2} 2 |D_x \psi_t|^2 + |\psi_t \wedge \psi_x|^2\ dx.
 \end{equation}
Putting this all together, we see that if we define
\begin{equation}\label{esd-def}
\ESD(\Phi)(s) := \| \psi_s(s) \|_{L^2_x(\R^2)}^2 + \| D_x \psi_t(s) \|_{L^2_x(\R^2)}^2 + \frac{1}{2} \| \psi_t \wedge \psi_x(s) \|_{L^2_x(\R^2)}^2
\end{equation}
then we have \eqref{energy-ident}; note that this quantity is $SO(m)$-invariant and thus independent of the choice of $\Psi_{s,t,x}$.  Also, it is clear that $\ESD(\Phi)$ is continuous in $s$.

It is not difficult to verify the symmetries \eqref{spacetrans-esd}-\eqref{dil-esd}.  To complete the proof of Proposition \ref{esd-prop}, it remains to verify the convergence property (iii).  Let $\Phi^{(n)}$ be a sequence in $\S$ with $\iota(\Phi^{(n)})$ Cauchy in $\Energy$, then by construction of $\Energy$, we can find $\Psi^{(n)}_{s,t,x} \in \SC(\Phi^{(n)})$ which is Cauchy in ${\mathcal L}$, thus
$$ \int_0^\infty \|\psi^{(n)-(n')}_s(s) \|_{L^2_x(\R^2)}^2\ ds = o_{n,n' \to \infty}(1)$$
and
\begin{equation}\label{siren}
 \| \psi^{(n)-(n')}_t(0)  \|_{L^2_x(\R^2)}^2 = o_{n,n' \to \infty}(1),
\end{equation}
where we write $\psi^{(n)-(n')}_s$ for $\psi^{(n)}_s - \psi^{(n')}_s$, etc., and $o_{n,n' \to \infty}(1)$ denotes a quantity that goes to zero as $n,n' \to \infty$.  

Fix a compact subset $S \subset (0,+\infty)$.  From the proof of \cite[Proposition 5.5]{tao:heatwave2} we see that $ \|\psi^{(n)-(n')}_s(s) \|_{L^2_x(\R^2)} = o_{n,n' \to \infty}(1)$ uniformly for $s \in S$.  Also, if we define
\begin{equation}\label{hodie}
h(s) := \|\psi^{(n)-(n')}_x\|_{L^\infty_x(\R^2)} + \|A^{(n)-(n')}_x\|_{L^\infty_x(\R^2)}
\end{equation}
then from the proof of \cite[Proposition 5.5]{tao:heatwave2} we have
\begin{equation}\label{sousa}
\sup_{s>0} s^{1/2} h(s) + \int_0^\infty h(s)^2\ ds = o_{n,n' \to \infty}(1).
\end{equation}
In particular $h(s) = o_{n,n' \to \infty}(1)$ uniformly for $s \in S$.  Using these estimates (and noting from \cite[Proposition 4.3, Corollary 4.6]{tao:heatwave2} that $\psi_x(s), A_x(s), \psi_t(s), \partial_x \psi_t(s)$ are bounded in $L^2_x$ and $L^\infty_x$ uniformly in $s \in S$) we see that we will be done as soon as we can show that
$$ \| \psi^{(n)-(n')}_t(s) \|_{L^2_x(\R^2)}, \| \partial_x \psi^{(n)-(n')}_t(s) \|_{L^2_x(\R^2)} = o_{n,n' \to \infty}(1)$$
uniformly in $s$.  From Proposition \ref{corbound} (or Lemma \ref{uheat-est}) we know that $\| \partial_x^2 \psi^{(n)}_t(s) \|_{L^2_x(\R^2)}$ is bounded uniformly in $n$ and for $s \in S$, so it in fact suffices from the Gagliardo-Nirenberg inequality to prove the first bound:
$$ \| \psi^{(n)-(n')}_t(s) \|_{L^2_x(\R^2)} = o_{n,n' \to \infty}(1).$$
Now, by subtracting the covariant heat equations \eqref{covheat} for $\psi^{(n)}_t, \psi^{(n')}_t$ we see that
\begin{equation}\label{psint}
 \partial_s \psi^{(n)-(n')}_t = D^{(n)}_i D^{(n)}_i \psi^{(n)-(n')}_t - (\psi^{(n)-(n')}_t \wedge \psi^{(n)}_i) \psi^{(n)}_i + E 
\end{equation}
where the error term $E$ has the magnitude bound
\begin{align*}
 |E| &\lesssim |\partial_x A^{(n)-(n')}_x| |\psi^{(n)+(n')}_t| + |A^{(n)-(n')}_x| |\partial_x \psi^{(n)+(n')}_t| \\
 &\quad + 
|A^{(n)-(n')}_x| |A^{(n)+(n')}_x| |\psi^{(n)+(n')}_t| + |\psi^{(n)-(n')}_x| |\psi^{(n)+(n')}_x| |\psi^{(n)+(n')}_t|
\end{align*}
where we write $\psi^{(n)+(n')}_t$ for $(\psi^{(n)}_t, \psi^{(n')}_t)$, etc.  Multiplying \eqref{psint} against $\psi^{(n)-(n')}_t$ and integrating, and discarding some negative terms (cf. the proof of \cite[Lemma 4.8]{tao:heatwave2}), we see that
$$ \partial_s  \|\psi^{(n)-(n')}_t \|_{L^2_x(\R^2)}^2 \leq 2 \| \psi^{(n)-(n')}_t \|_{L^2_x(\R^2)} \| E \|_{L^2_x(\R^2)}$$
and thus
$$ \partial_s  \|\psi^{(n)-(n')}_t \|_{L^2_x(\R^2)} \leq \| E \|_{L^2_x(\R^2)}.$$
By \eqref{siren}, we will thus be done as soon as we show
$$ \int_0^\infty \| E(s) \|_{L^2_x(\R^2)}\ ds = o_{n,n' \to \infty}(1).$$
From \eqref{hodie} we have
\begin{align*}
\| E(s) \|_{L^2_x(\R^2)} &\lesssim \| \partial_x A^{(n)-(n')}_x \|_{L^4_x(\R^2)} \| \psi^{(n)+(n')}_t\|_{L^4_x(\R^2)} + h(s) \times \\
&\quad ( \| \partial_x \psi^{(n)+(n')}_t\|_{L^2_x(\R^2)} + \|A^{(n)+(n')}_x\|_{L^\infty_x(\R^2)} \|\psi^{(n)+(n')}_t\|_{L^2_x(\R^2)}\\
&\quad + \|\psi^{(n)+(n')}_x\|_{L^\infty_x(\R^2)} \|\psi^{(n)+(n')}_t\|_{L^2_x(\R^2)} ).
\end{align*}
The terms involving $h(s)$ can be controlled using Cauchy-Schwarz, \eqref{sousa}, and Proposition \ref{corbound}.  As for the first term, observe from Lemma \ref{uheat-est} that
$$ \int_0^\infty s^{-1/2} \| \psi^{(n)+(n')}_t\|_{L^4_x(\R^2)}^2\ ds$$
is bounded uniformly in $n,n'$, so by Cauchy-Schwarz it suffices to show that
$$ \int_0^\infty s^{1/2} \| \partial_x A^{(n)-(n')}_x\|_{L^4_x(\R^2)}^2\ ds = o_{n,n' \to \infty}(1).$$
But from Proposition \ref{corbound} we have
$$ \int_0^\infty s \| \partial^2_x A^{(n)-(n')}_x\|_{L^2_x(\R^2)}^2\ ds $$
bounded uniformly in $n,n'$, so the claim follows from \eqref{sousa}, Cauchy-Schwarz, and the Gagliardo-Nirenberg inequality.  This completes the proof of Proposition \ref{esd-prop}.

\subsection{Compactness}\label{compact-sec}

In this section we prove Proposition \ref{precom}.  Fix $C: \R^+ \to \R^+$ and $E$ as in that proposition.  From Proposition \ref{esd-prop}, it suffices to show that every sequence $\iota(\phi^{(n)}_0, \phi^{(n)}_1)$ in (i) has a subsequence that converges in $\Energy$, as one can then take $K$ to be the set of all limits of such subsequences.  

Fix $\iota(\phi^{(n)}_0, \phi^{(n)}_1)$ as in (i).  We select $\Psi^{(n)}_{s,t,x} \in \SC(\phi^{(n)}_0, \phi^{(n)}_1)$ arbitrarily\footnote{The ambiguity in selecting $\Psi^{(n)}_{s,t,x}$ is described by the compact group $SO(m)$, and is thus irrelevant for the compactness theory.}.  We construct the non-negative measures $\mu^{(n)}$ on $\R^+ \times \R^2$ by the formula
$$ \mu^{(n)} := |\psi^{(n)}_s(s,x)|^2\ dx ds.$$
From \eqref{energy-ident} we see that $\mu^{(n)}$ has total mass at most $E$, thus the measures are uniformly finite in $n$.

\begin{lemma}[Tightness for $\psi^{(n)}_s$]\label{tight}  The $\mu^{(n)}$ form a tight sequence of measures; thus for every $\eps > 0$ there exists a compact subset $K$ of $\R^+ \times \R^2$ such that $\mu^{(n)}( \R^+ \times \R^2 \backslash K ) \leq \eps$ for all sufficiently large $n$.
\end{lemma}

\begin{proof}  From the hypothesis \eqref{esd-n} and \eqref{esd-def} we see that for every $\eps > 0$ one has
$$ \mu^{(n)}( [0,1/C(\eps)] \times \R^2 ) + \mu^{(n)}( [C(\eps),\infty) \times \R^2 ) \leq \eps$$
for $n$ sufficiently large.  Thus we already have tightness in the $s$ variable.  To obtain tightness in the $x$ variable, we need to exploit \eqref{locale}.  This will require an approximate finite speed of propagation result for the harmonic map heat flow.

From \eqref{psit-eq} and Lemma \ref{uheat-est} we have the mass diffusion identity
$$
\partial_s |\psi^{(n)}_x|^2 = \Delta |\psi^{(n)}_x|^2 - 2 |D_x \psi^{(n)}_x|^2 - |\psi_x \wedge \psi_x|^2.$$
Integrating this against a smooth non-negative compactly supported cutoff function $\chi: \R^2 \to \R^+$ to be chosen later, we obtain
$$
\partial_s \int_{\R^2} |\psi^{(n)}_x(s,x)|^2 \chi(x)\ dx \leq \int_{\R^2} |\psi^{(n)}_x(s,x)|^2 \Delta \chi(x)\ dx - 
2 \int_{\R^2} |D_x \psi^{(n)}_x(s,x)|^2 \chi(x)\ dx.$$
From Proposition \ref{corbound} we know that $\psi^{(n)}_x(s)$ is bounded in $L^2$ by $O_E(1)$.
From the fundamental theorem of calculus, and observing that $|\psi^{(n)}_s| \lesssim |D_x \psi^{(n)}_x|$, 
we conclude that
$$ \int_0^S \int_{\R^2} |\psi^{(n)}_s(s,x)|^2 \chi(x)\ dx ds \lesssim_E \int_{\R^2} |\psi^{(n)}_x(0,x)|^2 \chi(x)\ dx + S \| \Delta \chi \|_{L^\infty(\R^2)} $$
for any $S > 0$.  If we then let $\chi$ be a smooth cutoff to the region $R \leq |x| \leq R'$ that equals $1$ when $2R \leq |x| \leq R'/2$, then lets $R' \to \infty$, we obtain the bound
$$ \int_0^S \int_{|x| \geq 2R} |\psi^{(n)}_s(s,x)|^2\ dx ds \lesssim_E \int_{|x| \geq R} |\psi^{(n)}_x(0,x)|^2\ dx + S / R^2.$$
Applying \eqref{locale}, we see that for every $\eps' > 0$ there exists an $R > 0$ (depending on $C, \eps, \eps'$) such that
$$ \int_{1/C(\eps)}^{C(\eps)} \int_{|x| \geq 2R} |\psi^{(n)}_s(s,x)|^2\ dx ds \lesssim_E \eps'$$
for sufficiently large $n$, and the tightness claim follows.
\end{proof}

From Proposition \ref{corbound}, we see that the $\psi^{(n)}_s$ are equicontinuous on any compact subset of $\R^+ \times \R^2$, and thus by the Arzel\'a-Ascoli theorem one can (after passing to a subsequence) assume that the $\psi^{(n)}_s$ converge locally uniformly on compact sets.  Combining this with Lemma \ref{tight} we conclude that the $\psi^{(n)}_s$ are convergent in $L^2( \R^+ \times \R^2, dx ds)$.

From \eqref{psipsi}, we see that in order to conclude the proof that $\iota(\phi^{(n)}_0, \phi^{(n)}_1)$ is Cauchy in $\Energy$ (after passing to a subsequence), it suffices to show that the $\psi^{(n)}_t(0)$ are convergent in $L^2_x(\R^2)$ (again after passing to a subsequence).

From Proposition \ref{corbound}, we know that the non-negative measures $|\psi^{(n)}_t(s,x)|^2\ dx$ are uniformly finite in $n$ and $s \geq 0$.  For $s>0$ we also have tightness:

\begin{lemma}[Tightness for $\psi^{(n)}_t$]\label{tito}  For any fixed $s > 0$, the measures $|\psi^{(n)}_t(s,x)|^2\ dx$ are a tight sequence of measures in $n$.
\end{lemma}

\begin{proof}  This is another application of approximate finite speed of propagation.  Using \eqref{psit-eq} and Lemma \ref{uheat-est} as before we have
$$ \partial_s |\psi^{(n)}_t|^2 = \Delta |\psi^{(n)}_t|^2 - 2 |D_x \psi^{(n)}_t|^2 - |\psi_t \wedge \psi_x|^2.$$
Integrating this against a cutoff $\chi$ and discarding some negative terms we have
$$
\partial_s \int_{\R^2} |\psi^{(n)}_t(s,x)|^2 \chi(x)\ dx \leq \int_{\R^2} |\psi^{(n)}_t(s,x)|^2 \Delta \chi(x)\ dx,$$
and thus from the fundamental theorem of calculus (and Proposition \ref{corbound})
$$
\int_{\R^2} |\psi^{(n)}_t(s,x)|^2 \chi(x)\ dx \lesssim \int_{\R^2} |\psi^{(n)}_t(0,x)|^2 \chi(x)\ dx + s \| \Delta \chi \|_{L^\infty_x(\R^2)}.$$
Selecting $\chi$ as in the proof of Lemma \ref{tight} and applying the same limiting argument, we conclude that
$$
\int_{|x| \geq 2R} |\psi^{(n)}_t(s,x)|^2\ dx \lesssim \int_{|x| \geq R} |\psi^{(n)}_t(0,x)|^2 \chi(x)\ dx + s / R^2.$$
The tightness then follows from \eqref{locale}.
\end{proof}

By Proposition \ref{corbound}, the $\psi^{(n)}_t(s)$ are equicontinuous on compact subsets of $\R^2$ for any fixed $s$.  Applying the Arzel\'a-Ascoli theorem and Lemma \ref{tito}, we may thus assume (after passing to a subsequence and diagonalising) that the $\psi^{(n)}_t(s)$ converge in $L^2_x(\R^2)$ for any fixed rational $s > 0$.

To pass from $L^2$ convergence for $s > 0$ to $L^2$ convergence for $s=0$ we use

\begin{lemma}[Stability for $\psi^{(n)}$]   For every $\eps > 0$ there exists a rational $s_\eps > 0$ such that $\int_{\R^2} |\psi^{(n)}_t(0,x) - \psi^{(n)}_t(s_\eps,x)|^2\ dx \lesssim_E \eps$ for all sufficiently large $n$.
\end{lemma}

\begin{proof}  Write $\Psi^{(n)}_t(x) := \psi^{(n)}_t(s_\eps,x)$.
From \eqref{psits}, \eqref{esd-def} we have
$$ \int_{\R^2} |\Psi^{(n)}_t|^2\ dx = \int_{\R^2} |\psi^{(n)}_t(0,x)|^2\ dx + O( \int_0^{s_\eps} \ESD(\phi^{(n)}_0, \phi^{(n)}_1)(s)\ ds ).$$
From \eqref{esd-n} we thus see that if we choose $s_\eps$ small enough (depending on $\eps, C$), then
$$ \int_{\R^2} |\Psi^{(n)}_t|^2\ dx = \int_{\R^2} |\psi^{(n)}_t(0,x)|^2\ dx + O( \eps )$$
for all sufficiently large $n$.  By the cosine rule, it will thus suffice to show that
$$ \int_{\R^2} \psi^{(n)}_t(0,x) \cdot \Psi^{(n)}_t(x)\ dx = \int_{\R^2} |\Psi^{(n)}_t|^2\ dx + O_E( \eps )$$
for all sufficiently large $n$.  Write 
$$ F^{(n)}(s) := \int_{\R^2} \psi^{(n)}_t(s,x) \cdot \Psi^{(n)}_t(x)\ dx,$$
then by the fundamental theorem of calculus it suffices to show that
$$ \int_0^{s_\eps} |\partial_s F^{(n)}(s)|\ ds \ll_E \eps$$
for all sufficiently large $n$.

Applying \eqref{psit-eq} and integrating by parts, we obtain
$$
|\partial_s F^{(n)}(s)| \lesssim
\int_{\R^2} |D_x \psi^{(n)}_t| |D_x \Psi^{(n)}_t| + |\psi^{(n)}_x|^2 |\psi^{(n)}_t| |\Psi^{(n)}_t|\ dx.$$
Applying H\"older's inequality and \eqref{esd-def} we conclude
\begin{align*}
|\partial_s F^{(n)}(s)| &\lesssim \ESD(\phi^{(n)}_0, \phi^{(n)}_1)(s)^{1/2} \| D_x \Psi^{(n)}_t \|_{L^2_x(\R^2)}\\
&\quad + \| \psi^{(n)}_t(s) \|_{L^4_x(\R^2)} \| \psi^{(n)}_x \|_{L^\infty_x(\R^2)}^{1/2} \| \psi^{(n)}_x \|_{L^2_x(\R^2)}^{3/2} \| \Psi^{(n)}_t \|_{L^\infty_x(\R^2)}.
\end{align*}
From Proposition \ref{corbound} we have the estimates
\begin{align*}
\| \Psi^{(n)}_t \|_{L^\infty_x(\R^2)}, \| D_x \Psi^{(n)}_t \|_{L^2_x(\R^2)} &\lesssim_E s_\eps^{-1/2} \\
\| \psi^{(n)}_x(s) \|_{L^2_x(\R^2)} &\lesssim_E 1 \\
\| \psi^{(n)}_x(s) \|_{L^\infty_x(\R^2)} &\lesssim_E s^{-1/2}.
\end{align*}
Also, from the diamagnetic inequality $|\partial_x |\psi^{(n)}_t|| \leq |D_x \psi^{(n)}_t|$ and \eqref{esd-def} we have
$$ \| \partial_x |\psi^{(n)}_t| \|_{L^2_x(\R^2)} \lesssim \ESD(\phi^{(n)}_0, \phi^{(n)}_1)(s)^{1/2};$$
meanwhile, from Proposition \ref{corbound} we have $\| |\psi^{(n)}_t| \|_{L^2_x(\R^2)} \lesssim_E 1$.  From the Gagliardo-Nirenberg inequality we conclude
$$ \| \psi^{(n)}_t \|_{L^4_x(\R^2)} \lesssim_E \ESD(\phi^{(n)}_0, \phi^{(n)}_1)(s)^{1/4}.$$
Putting all these estimates together, we conclude
$$
|\partial_s F^{(n)}(s)| \lesssim_E s_\eps^{-1/2} \ESD(\phi^{(n)}_0, \phi^{(n)}_1)(s)^{1/2} 
+ s_\eps^{-1/2} s^{-1/4} \ESD(\phi^{(n)}_0, \phi^{(n)}_1)(s)^{1/4}.$$
The claim then follows from \eqref{esd-n} and H\"older's inequality.
\end{proof}

Since the $\psi^{(n)}_t(s_\eps,x)$ were already convergent in $L^2_x(\R^2)$ for each $\eps > 0$, we conclude from the triangle inequality that the $\psi^{(n)}_t(0,x)$ are Cauchy in $L^2_x(\R^2)$, and the claim follows.  The proof of Proposition \ref{precom} is now complete.

\end{document}